\documentclass[<opyions>]{elsarticle}
\usepackage{amssymb}
\usepackage{graphicx}
\usepackage{epstopdf}
\usepackage{subfigure}
\usepackage{extarrows}
\usepackage{fancyhdr}
\usepackage{amsthm}
\usepackage{extarrows}
\usepackage{caption}
\usepackage{float}
\usepackage{algorithm}
\usepackage{algorithmic}
\usepackage{multirow}
\usepackage{amsmath}
\usepackage[numbers]{natbib}
\usepackage{graphicx}
\usepackage{amsmath} 
\usepackage{amsfonts} 
\usepackage{mathrsfs} 
\usepackage{booktabs} 
\usepackage{color} 

\usepackage{color}
\usepackage{hyperref}

\usepackage[toc,page,title,titletoc,header]{appendix}
\usepackage{appendix}
\theoremstyle{definition}

\theoremstyle{definition}
\usepackage{hyperref}
\usepackage{latexsym, bm}
\usepackage{fancyhdr}
\usepackage{mathrsfs}
\usepackage{wasysym}
\usepackage{float}
\usepackage{titlesec}
\usepackage{pgfplots}
\usepackage{tikz}
\usepackage{subfigure}
\usepackage{natbib}
\biboptions{numbers,sort&compress}

\usetikzlibrary{arrows,shapes,positioning}
\usetikzlibrary{decorations.markings}
\tikzstyle arrowstyle=[scale=1]
\tikzstyle directed=[postaction={decorate,decoration={markings,
    mark=at position .65 with {\arrow[arrowstyle]{stealth}}}}]
\tikzstyle reverse directed=[postaction={decorate,decoration={markings,
    mark=at position .65 with {\arrowreversed[arrowstyle]{stealth};}}}]

\newtheorem{lemma}{Lemma}[section]
\newtheorem{theorem}{Theorem}[section]

\newtheorem{example}{Example}[section]

\allowdisplaybreaks
\begin{document}
	
\begin{frontmatter}
\title{{\bf A Linearized and structure-preserving mixed virtual element method for the extended Fisher-Kolmogorov equation
}}

\author{Zhen Guan\corref{cor1}}
\ead{zhenguan1993@foxmail.com}

\author{Xianxian Cao}
\author{Houchao Zhang}
\author{Junjun Wang}

\cortext[cor1]{Corresponding author.}

\address{School of Mathematics and Statistics, Pingdingshan University, Pingdingshan, 467000, China}

\begin{abstract}
In thsi paper, based on the leap-frog discretization in time and the mixed virtual element discretization in space, we developed a linearized and structure-preserving numerical algorithm. The main contributions of this work lie in that we not only provide a rigorous proof of the energy dissipation property of the fully discrete numerical scheme, but also establish the unconditionally optimal convergence analysis by means of a inverse inequality. The core of the proof lies in the classified discussion of the relationship between \(\tau\) and $h$. Finally, two numerical examples are provided to validate the correctness of the theoretical analysis as well as the energy dissipation property of the proposed scheme.
\end{abstract}

\end{frontmatter}

\thispagestyle{empty}

\numberwithin{equation}{section}
\section{Introduction}\label{section01}
In this work, we develop a structure-preserving numerical scheme for solving the following extended Fisher-Kolmogorov equation (EFK) in polygonal mesh
\begin{align}
	&u_t + \gamma \Delta^2 u - \Delta u + u^3-u = 0, \quad (\boldsymbol{x},t) \in \Omega \times (0,T], \label{202606142344}\\
	&u = \Delta u = 0, \quad (\boldsymbol{x},t) \in \partial\Omega \times (0,T], \label{202606142345}\\
	&u(\boldsymbol{x},0) = u_0(\boldsymbol{x}), \quad \boldsymbol{x} \in \Omega,\label{202606142346}
\end{align}
where $\gamma$ is a positive constant, $u(\boldsymbol{x},t) $ and $u_0(\boldsymbol{x})$ are real-valued functions, $\Omega\subset \mathbb{R}^2$ is a bounded convex polygonal domain, $\partial\Omega$ is the boundary of $\Omega$ and $\Delta$ is the Laplace 
operator. 

The EFK equation is derived by adding a fourth-order term to the standard Fisher-Kolmogorov (FK) equation, which constitutes an important class of nonlinear fourth-order evolution equations.
It has important applications in population genetics \cite{Aronson}, domain wall propagation in liquid crystals \cite{Guo}, and the growth process of primary brain tumors \cite{Coullet}. Due to the high cost of obtaining exact solutions to nonlinear partial differential equations, numerous effective numerical methods have been developed by scholars in recent years. For example, Danumjaya and Pani \cite{Danumjaya} developed a first-order fully discrete numerical algorithm by adopting the \(C^1\)-conforming finite element method and the implicit Euler scheme, and rigorously proved the convergence of the proposed numerical algorithm. Liu and Yin \cite{Liu} derived a parameter-free discontinuous Galerkin algorithm by adopting scalar auxiliary variable time discretization for solving a class of fourth-order gradient flow problems. Boujlida et al. \cite{Boujlida} proposed a three-layer compact difference scheme for solving the one-dimensional EFK equation, and derived the unique solvability and convergence of the scheme via the energy analysis method. Kumar and Natara \cite{Kumar} 
adopted the Euler and Crank-Nicolson numerical schemes to develop a hybrid high-order discretization method for solving the nonlinear EFK and FK equations, and analyzed the corresponding temporal and spatial error estimates. Chauhan and Chaudhary \cite{Chauhan} investigated the space-time isogeometric method for a class of linear fourth-order evolution problems. The core idea is to introduce a auxiliary variable to decompose the fourth-order problem into a system of second-order equations. For the extended Fisher–Kolmogorov equation with clamped boundary conditions, Das and Nataraj \cite{Das} performed spatial discretization via the lowest-order nonstandard finite element method and implemented temporal discretization using the backward Euler scheme. Yang et al. \cite{Yang} constructed a Crank–Nicolson mixed Galerkin scheme for the two-dimensional EFK equation, and conducted theoretical analyses of its convergence and superconvergence errors. Abbaszadeh et al. \cite{Abbaszadeh} adopted the interpolating element-free Galerkin method to solve the nonlinear EFK equation. Shi and Zhang \cite{Shi} investigated the superconvergence property of an energy-dissipative BDF2 scheme combined with anisotropic nonconforming finite element method (FEM) for solving the fourth-order singularly perturbed Bi-flux diffusion model. Wang et al. \cite{Wang} constructed a second-order accurate numerical scheme for the EFK equation, where Fourier spectral collocation was adopted for spatial discretization and stabilized Runge–Kutta–Munthe-Kaas-2e (RKMK2e) served as the temporal approximation method. They further carried out theoretical analysis on the global-in-time energy stability and convergence of the fully discrete scheme, and performed numerical examples to verify the theoretical results. Fu et al. \cite{Fu} investigated the ultraweak local discontinuous Galerkin (UWLDG) method to solve the initial boundary value problem of the EFK equation. Pei et al. \cite{Pei} developed a unified nonconforming virtual element framework incorporating \(C^0\) and full non-conforming virtual elements to discretize the fourth-order EFK equation. We only summarize recent published work on EFK numerical schemes above. Readers interested in this topic can consult the bibliographies of the listed papers for further related computational research.

To the best of our knowledge, existing temporal discretization schemes in the literature can be roughly divided into two categories: fully implicit approaches and linearization-based strategies. Although fully implicit schemes can preserve certain inherent structures of the continuous equation, they require solving a system of nonlinear equations at each time step, which introduces substantial computational overhead. 
This consideration motivates the development of linearized numerical schemes in this work. As reviewed in previous studies, most existing linearized numerical schemes fail to preserve the intrinsic physical structures of the original equation. In addition, the majority of available methods are only applicable to triangular and quadrilateral meshes. Against this background, this work aims to construct a structure-preserving and linearized mixed virtual element algorithm, which is capable of simulating physical phenomena on more complex computational domains. 

To construct the efficient numerical scheme, we first introduce a new variable $v=-\Delta u$ and reformulate the original equation as a second-order system. Subsequently, temporal discretization is implemented via the leap-frog scheme. The core novelty of our work lies in the discrete approximation of cubic nonlinear terms, which differs from the treatments adopted in prior literature, namely
 $$u(t_{n-1})^3-u(t_{n-1}) \approx u(t_{n-1})^2\left(\frac{u(t_{n})+u(t_{n-2})}{2}\right)-\frac{u(t_{n})+u(t_{n-2})}{2},$$
which is critical to maintaining the energy decay property. Furthermore, the discretization scheme \eqref{202606190819} for initial layers also exerts a substantial influence on energy reduction. Another novelty of this work lies in the classification discussion of two cases \(\tau\leq h\) and \(\tau>h\), through which we derive the unconditionally optimal convergence of the fully discrete numerical scheme. As far as the authors are aware, this analytical technique is introduced to the virtual element method for the first time. A key inequality employed throughout the proof is the inverse inequality satisfied by virtual element functions. This result is well established in the virtual element community.

This paper is structured as follows. In Section \ref{section2}, we first introduce basic notations for the virtual element method and then present the detailed derivation of the fully discrete numerical scheme. Section \ref{section3} is devoted to the theoretical analysis of the proposed scheme, including existence and uniqueness, boundedness, energy decay property as well as optimal convergence. Section \ref{section4} presents two numerical examples with known and unknown solutions to verify the convergence accuracy and energy decay property of the proposed scheme, respectively. Section \ref{section5} concludes this paper and outlines possible research directions.

\section{Derivation of the structure-preserving mixed virtual element method}\label{section2}
\subsection{The variational formulation}
We follow conventional notation for Sobolev spaces and their associated norms throughout this work \cite{Ciarlet1978, Ciarlet2013, Zenisek2005}.
For any open subset $D \subset \Omega$, $|\cdot|_{m,p,D}$ stands for the Sobolev seminorm on $W^{m,p}(D)$, while $\|\cdot\|_{m,p,D}$ denotes its full norm.
In the special case $p=2$, we write $H^m(D) = W^{m,2}(D)$, with shortened seminorm $|\cdot|_{m,D}$ and norm $\|\cdot\|_{m,D}$.
The space $H_0^m(D)$ is defined as the completion of $C_0^\infty(\Omega)$ under the norm $\|\cdot\|_{m,D}$.
The symbol $(\cdot,\cdot)_D$ signifies the $L^2(D)$ inner product, and $\|\cdot\|_D$ is the induced $L^2$ norm on this Hilbert space.
Whenever $D$ coincides with the global domain $\Omega$, we drop the subscript $D$ from all preceding seminorms, norms and inner products. It should be noted that the above notation also applies to vector-valued spaces. Moreover, the constant $C$ may take different values at different occurrences. However, constants with subscripts, such as \(C_1,C_2\), stand for a fixed positive constant. Lastly, given a strongly measurable mapping $v: (0,T) \to X$ with $X$ a real Banach space, the Bochner space is formulated by
\[
L^p((0,T);X) = \big\{v \;\big|\; \|v\|_{L^p((0,T);X)} < \infty\big\}.
\]

By introducing a new variable $v = -\Delta u$, the original equation \eqref{202606142344}-\eqref{202606142346} can be restated in an equivalent form as
\begin{align}
	&u_t - \gamma \Delta v - \Delta u + u^3-u = 0, \quad (\boldsymbol{x},t) \in \Omega \times (0,T], \label{202606142347}\\
	& v = -\Delta u,\quad (\boldsymbol{x},t) \in \Omega \times (0,T], \label{202606142350}\\
	&u = v = 0, \quad (\boldsymbol{x},t) \in \partial\Omega \times (0,T], \label{202606142348}\\
	&u(\boldsymbol{x},0) = u_0(\boldsymbol{x}),\quad \boldsymbol{x} \in \Omega.\label{202606142349}
\end{align}

With the notations specified above, we are able to derive the variational formulation satisfied by \eqref{202606142347}-\eqref{202606142349}: Find $u(t),v(t)\in H_0^1(\Omega), t\in(0,T]$ such that 
\begin{align}
	&(u_t,\phi) + \gamma (\nabla v,\nabla \phi) + (\nabla u,\nabla\phi) + (u^3-u,\phi) = 0, \quad \forall\phi\in H_0^1(\Omega), \label{202606151417}\\
	& (v,\psi) = (\nabla u,\nabla \psi),\quad \forall\psi\in H_0^1(\Omega), \label{202606151418}
\end{align}
with the initial condition $u(0)=u_0$. 

To derive the energy dissipation property of the equation, we take the first-order time derivative of both sides of \eqref{202606151418}.
\begin{align}
(v_t,\psi) = (\nabla u_t,\nabla \psi),\quad \forall\psi\in H_0^1(\Omega). \label{202606180933}
\end{align}
Substitute \(\phi = u_t\) into \eqref{202606151417} and \(\psi = v\) into \eqref{202606180933}, respectively, then multiply the second equation by $\gamma$ and sum the two resulting equations to get
\begin{align}
\frac{d}{dt}E(t)=-\|u_{t}\|^2\leq 0,\notag
\end{align}
where the free energy $E(t)$ is defined as 
\begin{align}
E(t)=\frac{\gamma}{2}\|v\|^2+\frac{1}{2}\|\nabla u\|^2+\frac{1}{4}(u^2,u^2)-\frac{1}{2}\|u\|^2,\notag 
\end{align}
which implies that the energy of the system decays over time.

\subsection{Conforming virtual element terminology}
Let \(\mathcal{T}_h\) denote a family of polygonal meshes of the domain \(\Omega\). The mesh size $h$ is defined as
\(h = \max_{K\in\mathcal{T}_h} h_K,\) where \(h_K\) stands for the diameter of polygon $K$. We impose standard regularity assumptions on the mesh sequence \(\mathcal{T}_h\), i.e., there exists a constant \(\rho\in(0,1)\) satisfying

$\bullet$ for each element $K$ and any edge $E\subset\partial K$, it holds that $h_{E}\geq\varrho h_{K}$;

$\bullet$ every element $K$ is star-shaped with respect to a disc $B$ whose radius $\geq \varrho h_{K}$;

$\bullet$  the mesh is quasi-uniform in the sense that $h_{K}\geq \varrho h,~\forall K\in \mathcal{T}_h$.

The local elliptic projection \(\Pi_{K}^{1,k}\), which maps \(H^{1}(K)\) onto the polynomial space \(\mathbb{P}_{k}(K)\), is specified by
\begin{equation*}
	\left\{
	\begin{aligned}
		&(\nabla(\Pi^{1,k}_{K}v),\nabla q)_{K}=(\nabla v,\nabla q)_{K},\quad \forall q\in \mathbb{P}_{k}(K),\\
		&P_{0}^{K}(\Pi^{1,k}_{K}v)=P_{0}^{K}v.\\
	\end{aligned}
	\right.
\end{equation*}
Here \(P_0^K\) is a linear operator producing constant scalars, and its definition reads as follows:
\begin{equation*}
	\left\{
	\begin{aligned}
		&P_{0}^{K}(v)=\frac{1}{N_V} \sum_{i=1}^{N_V} v(V_i),\quad k=1,\\
		&P_{0}^{K}(v)=\frac{1}{|K|}\int_{ K}v\mathrm{d}\boldsymbol{x},\quad k\geq 2,\\
	\end{aligned}
	\right.
\end{equation*}
where \(N_{V}\) denotes the total number of vertices of the cell $K$, and each \(V_i\) stands for a vertex of $K$. 

Likewise, we introduce the local \(L^2\)-orthogonal projection \(\Pi_{K}^{0,k}: L^2(K) \to \mathbb{P}_k(K)\) via the following condition: for any function \(v\in L^2(K)\), the polynomial \(\Pi_{K}^{0,k}v\in \mathbb{P}_k(K)\) fulfills
$$(\Pi_{K}^{0,k} v, q)_{K} = (v, q)_{K},\quad \forall q \in \mathbb{P}_{k}(K).$$
Accordingly, the global \(L^2\)-orthogonal projection operator \(\Pi_{h}^{0,k}: L^2(\Omega) \to \mathbb{P}_k(\mathcal{T}_h)\) can be readily constructed, i.e., for all $v\in L^{2}(\Omega)$ and all $K\in \mathcal{T}_{h} $
$$(\Pi^{0,k}_{h}v)|_{K}=\Pi^{0,k}_{K}(v|_{K}),\quad$$
where $\mathbb{P}_{k}(\mathcal{T}_{h})$ is the broken polynomial space with polynomial degree no greater than $k$.

Given any mesh element $K$, we introduce the enriched conforming local virtual element space 
\begin{align}
W_{K}^{k}=\{v_{K}\in H^{1}(K): \Delta v_{K}\in \mathbb{P}_{k}(K), v_{K}|_{\partial K}\in \mathbb{B}_{k}(\partial K)\},\notag 
\end{align} 
where the boundary element space $\mathbb{B}_{k}(\partial K)$ is defined as
\begin{align}
 \mathbb{B}_{k}(\partial K)=\{v\in C(\partial K):v|_{E}\in \mathbb{P}_{k}(E),~\forall E\subset\partial K\}.\notag 
\end{align}

Now, we define the enhanced local conforming virtual element space defined by 
\begin{align}
	V_{K}^{k}=\{v_{K}\in W_{K}^{k}: (v_{K},q)_{K}=(\Pi^{1,k}_{K}v_{K},q)_{K},~\forall q\in \mathcal{M}_{k}(K)-\mathcal{M}_{k-2}(K)
	\},\notag 
\end{align}
where the symbol \(\mathcal{M}_{k}(K)\) stands for the collection of scaled monomials, which are defined by 
\begin{align}
	\mathcal{M}_{k}(K)=\bigg\{\bigg(\frac{\boldsymbol{x}-\boldsymbol{x}_{K}}{h_{K}}\bigg)^{\boldsymbol{s}}:|\boldsymbol{s}|\leq k\bigg\},\notag 
\end{align}
where \(\boldsymbol{s}=(s_1,s_2)\) denotes a multi-index, \(|\boldsymbol{s}|=s_1+s_2\) and \(\boldsymbol{x}^{\boldsymbol{s}}=x_1^{s_1}x_2^{s_2}\). The symbol \(\boldsymbol{x}_K\) stands for the centroid of element $K$. By an analogous construction, we may define the scaled monomial set \(\mathcal{M}_{k}(E)\) for any edge $E$. 

The degrees of freedom associated with the aforementioned finite-dimensional space \(V_{K}^{k}\) may be selected as

$\bullet$ the values of $v_K\in V_{K}^{k}$ at the vertices,

$\bullet$ for $k\geq2$, the moments $\frac{1}{|E|}\int_{E}v_{K}m\mathrm{d}\boldsymbol{s}$ for each $m\in \mathcal{M}_{k-2}(E)$ and any $E\subset\partial K$,

$\bullet$ for $k\geq2$, the moments $\frac{1}{|K|}\int_{K}v_{K}m\mathrm{d}\boldsymbol{x}$ for each $m\in \mathcal{M}_{k-2}(K)$ in $K$.

For the sake of notational simplicity, we label the aforementioned degrees of freedom as $\phi_{1},\phi_{2},\cdots,\phi_{N_{K}}$, in which \(N_K\) equals the dimension of the finite-dimensional local space $V_{K}^{k}$ defined on every mesh cell $K$.

Using the local degrees of freedom listed above to couple each local space \(V_{K}^{k}\) together, we obtain the global conforming virtual element space \(V_{h}^{k}\) as follows:
\begin{align}
	V_{h}^{k}=\{v_{h}\in H_{0}^{1}(\Omega): v_{h}|_{K}\in V_{K}^{k},~\forall K\in\mathcal{T}_{h}\}.\notag 
\end{align}                                                 

\subsection{Fully discrete numerical scheme}
Let \(0 = t_0 < t_1 < \dots < t_N = T\) form a uniform partition of the temporal interval \([0, T]\), with uniform time step \(\tau=\frac{T}{N}\). For any function sequence \(\{u^n\}_{n=0}^N\) defined on the temporal grid (here \(u^n = u(t_n)\) stands for the value of $u$ at discrete time \(t_n\)), we define the notations for difference quotients as follows:
\begin{align*}
&D_\tau u^1=\frac{u^1-u^0}{\tau},\quad D_\tau u^n = \frac{u^n-u^{n-2}}{2\tau},\quad 2\leq n \leq N,\\
&\hat{u}^1=\frac{u^1+u^0}{2},\quad \hat{u}^n=\frac{u^n+u^{n-2}}{2},\quad 2\leq n \leq N.
\end{align*}

With all the above preliminaries in place, we now turn to introducing the fully discrete finite element algorithm: find \(u_h^n,v_h^n\in V_h^k\ (n\geq2)\) such that
\begin{align}
&m_h(D_\tau u_h^n,\phi_h)+\gamma a_h(\hat{v}_h^n,\phi_h)+a_h(\hat{u}_h^n,\phi_h)\notag\\
&\quad+((\Pi_{h}^{0,k}u_h^{n-1})^2\Pi_{h}^{0,k}\hat{u}_h^n,\Pi_{h}^{0,k}\phi_h)-m_h(\hat{u}_h^n,\phi_h)=0,\quad \forall \phi_h\in V_h^k,\label{202606160945}\\
&m_h(v_h^n,\psi_h)=a_h(u_h^n,\psi_h),\quad  \forall \psi_h\in V_h^k,\label{202606160946}
\end{align}
where the global bilinear forms \(m_h:V_{h}^{k}\times V_{h}^{k}\to \mathbb{R}\) and \(a_h:V_{h}^{k}\times V_{h}^{k}\to \mathbb{R}\) are given by the expressions below:
\begin{align}
 a_{h}(u_{h},v_{h})&=\sum\limits_{K\in\mathcal{T}_{h}} (\nabla\Pi^{1,k}_{K}(u_{h}|_{K}),\nabla\Pi^{1,k}_{K}(v_{h}|_{K}))\notag\\
 &\quad+\sum\limits_{i=1}^{N_K}\phi_{i}(u_{h}|_{K}-\Pi^{1,k}_{K}(u_{h}|_{K}))\phi_{i}(v_{h}|_{K}-\Pi^{1,k}_{K}(v_{h}|_{K})),\quad \forall u_{h},v_{h}\in V_{h}^{k},\notag\\
 m_{h}(u_{h},v_{h})&=\sum\limits_{K\in\mathcal{T}_{h}} (\nabla\Pi^{0,k}_{K}(u_{h}|_{K}),\nabla\Pi^{0,k}_{K}(v_{h}|_{K}))\notag\\
 &\quad+h_K^2\sum\limits_{i=1}^{N_K}\phi_{i}(u_{h}|_{K}-\Pi^{0,k}_{K}(u_{h}|_{K}))\phi_{i}(v_{h}|_{K}-\Pi^{0,k}_{K}(v_{h}|_{K})),\quad \forall u_{h},v_{h}\in V_{h}^{k}.\notag
\end{align}
As indicated in literature \cite{BdV2013}, the discrete bilinear form $a_{h}(\cdot,\cdot)$ possess the boundedness and stability properties given as 
\begin{align*}
&\alpha_{*}|v_{h}|_{1}^2\leq |v_{h}|^2_{a_h}:=a_{h}(v_{h},v_{h})\leq \alpha^{*}|v_{h}|_{1}^2,\quad\forall v_{h}\in V_{h}^{k},\\
&\beta_{*}\|v_{h}\|\leq \|v_{h}\|^2_{m_h}:=m_{h}(v_{h},v_{h})\leq \beta^{*}\|v_{h}\|,\quad\forall v_{h}\in V_{h}^{k}.
\end{align*}

To obtain the numerical solution for the first time level and retain energy dissipation, we use the following Crank-Nicolson discretization, i.e., find $u_h^1,v_h^1\in V_h^k$ such that 
\begin{align}
	&m_h(D_\tau u_h^1,\phi_h)+\gamma a_h(\hat{v}_h^1,\phi_h)+a_h(\hat{u}_h^1,\phi_h)\notag\\
	&\quad+((\Pi_{h}^{0,k}u_h^0)^2\Pi_{h}^{0,k}\hat{u}_h^1,\Pi_{h}^{0,k}\phi_h)-m_h(\hat{u}_h^1,\phi_h)=0,\quad \forall \phi_h\in V_h^k\label{202606161007},\\
	&m_h(v_h^1,\psi_h)=a_h(u_h^1,\psi_h),\quad  \forall \psi_h\in V_h^k,\label{202606161004}
\end{align}
where \(u_h^0\) is defined as \(R_h u(0)\), and \(v_h^0\) solves the subsequent variational equation:
\begin{align}
	m_h(v_h^0,\psi_h)=a_h(u_h^0,\psi_h),\quad\forall \psi_h\in V_h^k.\label{202606190819}
\end{align}
Here, \(R_h\) stands for the elliptic projection operator introduced in \eqref{202606190820}.

Furthermore, to facilitate subsequent theoretical analysis, we define the elliptic projection operator \(R_h: H_0^1(\Omega)\cap H^{k+1}(\Omega)\to V_{h}^k\) satisfying
\begin{align}
a_h(R_h u, v_h) = -(\Delta u, \Pi_{h}^{0,k}v_h), \quad \forall v_h \in V_{h}^k.\label{202606190820}
\end{align}
Based on the analysis given in \cite{BdV2013}, we derive the following error bound for the projection operator:
\begin{align}
\|u-R_hu\| \leq C h^{k+1} \|u\|_{k+1}, \quad \forall\, u\in H_0^1(\Omega) \cap H^{k+1}(\Omega).
\end{align}
\subsection{A collection of essential theoretical results}
\begin{lemma}\label{lemma2.1}
\text{(\cite{Zhao2025})} Suppose $V$ is a normed linear space with norm \(\|\cdot\|\), and \(v^0, v^1, \dots, v^N \in V\). For every integer $n$ such that \(1\le n\le N\), the subsequent inequality is valid:
\begin{align}
\|v^n\|\leq 2\sum\limits_{k=1}^{n}\|\hat{v}^k\|+\|v^0\|,\quad 1\leq n\leq N.\notag
\end{align}
\end{lemma}
\begin{lemma}\label{lemma3}
\text{(\cite{HeywoodRannacher1990})} Assume \(a \geq 0\), \(b > 0\), and let \(\{\eta_i\}_{i=1}^N\), \(\{\xi_i\}_{i=1}^N\) be two sequences consisting of nonnegative real numbers. If the inequality below is satisfied for every integer \(1 \leq n \leq N\):
\begin{align*}
\eta_n + \tau \sum_{i=1}^n \xi_i \leq a + b\tau \sum_{i=1}^n \eta_i.
\end{align*}
Furthermore, provided that \(\tau \leq \dfrac{1}{2b}\), the refined estimate
$$
\eta_n + \tau \sum_{i=1}^n \xi_i \leq a \exp\big(2bn\tau\big)
$$
holds for every integer \(1 \leq n \leq N.\)
\end{lemma}
\section{Analysis of the fully discrete scheme}\label{section3}
\subsection{Existence of the virtual element solution}
\begin{theorem}
When $\tau\leq1$, the fully discrete numerical algorithm  \eqref{202606160945}-\eqref{202606161004} has a unique solution.
\end{theorem}
\begin{proof}
For system of linear equations, the existence and uniqueness of solution is equivalent to the corresponding homogeneous linear system having only the trivial zero solution. Therefore, we only consider the associated homogeneous linear system in the subsequent proof. 

Since the numerical solution at the initial time is already given by elliptic projection, we can obtain the system of linear equations in
$u_h^1$ and $v_h^1$ derived from \eqref{202606161007}-\eqref{202606161004}. Consider 
its homogeneous one:
\begin{align}
	&\frac{1}{\tau}m_h(u_h^1,\phi_h)+\frac{\gamma}{2}a_h(v_h^1,\phi_h)+\frac{1}{2}a_h(u_h^1,\phi_h)\notag\\
	&\quad+\frac{1}{2}((\Pi_{h}^{0,k}u_h^0)^2\Pi_{h}^{0,k}u_h^1,\Pi_{h}^{0,k}\phi_h)-\frac{1}{2}m_h(u_h^1,\phi_h)=0,\quad \forall \phi_h\in V_h^k\label{202606161455},\\
	&\frac{1}{2}m_h(v_h^1,\psi_h)=\frac{1}{2}a_h(u_h^1,\psi_h),\quad  \forall \psi_h\in V_h^k.\label{202606161456}
\end{align}
Take \(\phi_h=u_h^1\) and \(\psi_h=v_h^1\) in \eqref{202606161455} and \eqref{202606161456}, respectively, multiply both sides of \eqref{202606161456} by the constant \(\gamma\), and add the two resulting equations together, we obtain that
\begin{align}
\frac{1}{\tau}\|u_h^1\|_{m_h}^2+\frac{1}{2}\|u_h^1\|_{a_h}^2+\frac{1}{2}\|(\Pi_{h}^{0,k}u_h^1)(\Pi_{h}^{0,k}u_h^0)\|^2-\frac{1}{2}\|u_h^1\|_{m_h}^2+\frac{\gamma}{2}\|v_h^1\|_{m_h}^2=0.\notag
\end{align}
Noting the condition \(\tau\leq 1\), we can immediately arrive at \(\|u_h^1\|_{a_h}=0\) and \(\|v_h^1\|_{m_h}=0\). Combined with the boundedness and stability of the discrete bilinear forms \(a_h(\cdot,\cdot)\) and \(m_h(\cdot,\cdot)\), this yields \(u_h^1=0\) and \(v_h^1=0\), respectively. 

Suppose that $u_h^{n-1},u_h^{n-2},v_h^{n-1},v_h^{n-2}$ have been uniquely determined. Then we have the system of linear equations \eqref{202606160945} and \eqref{202606160946} in $u_h^n,v_h^n$. Consider its homogeneous one:
\begin{align}
	&\frac{1}{2\tau}m_h(u_h^n,\phi_h)+\frac{\gamma}{2}a_h(v_h^n,\phi_h)+\frac{1}{2}a_h(u_h^n,\phi_h)\notag\\
	&\quad+\frac{1}{2}((\Pi_{h}^{0,k}u_h^{n-1})^2\Pi_{h}^{0,k}u_h^n,\Pi_{h}^{0,k}\phi_h)-\frac{1}{2}m_h(u_h^n,\phi_h)=0,\quad \forall \phi_h\in V_h^k,\label{202606161617}\\
	&\frac{1}{2}m_h(v_h^n,\psi_h)=\frac{1}{2}a_h(u_h^n,\psi_h),\quad\forall \psi_h\in V_h^k.\label{202606161618}
\end{align}
Substituting \(\phi_h=u_h^n\) into \eqref{202606161617} and \(\psi_h=\gamma v_h^n\) into \eqref{202606161618}, then summing the two resulting equations, we arrive at
\begin{align}
	\frac{1}{2\tau}\|u_h^n\|_{m_h}^2+\frac{1}{2}\|u_h^n\|_{a_h}^2+\frac{1}{2}\|(\Pi_{h}^{0,k}u_h^n)(\Pi_{h}^{0,k}u_h^{n-1})\|^2-\frac{1}{2}\|u_h^n\|_{m_h}^2+\frac{\gamma}{2}\|v_h^n\|_{m_h}^2=0.\notag
\end{align}
This yields \(u_h^n = 0\) and \(v_h^n = 0\) provided that \(\tau \leq 1\). We complete the proof as desired.
\end{proof}
\subsection{Structure-preserving property}
\begin{theorem}
The fully discrete numerical scheme \eqref{202606160945}–\eqref{202606161004} satisfies the following energy decay property:
$$
E_h^{n}\leq E_h^{n-1},\quad 1\leq n\leq N,
$$	
where the discrete energy \(E_h^n\) with \(0\le n\le N\) is defined by
\begin{align}
E_h^{n} &= \frac{\gamma}{4}(\|v_h^n\|_{m_h}^2+\|v_h^{n-1}\|_{m_h}^2)+\frac{1}{4}(\| u_h^n\|_{a_h}^2+\|u_h^{n-1}\|_{a_h}^2)\notag\\
&\quad+\frac{1}{4}((\Pi_{h}^{0,k}u_h^n)^2,(\Pi_{h}^{0,k}u_h^{n-1})^2)-\frac{1}{4}(\|u_h^n\|_{m_h}^2+\|u_h^{n-1}\|_{m_h}^2),\quad 1\leq n\leq N,\notag\\
E_h^{0}&=\frac{\gamma}{2}\|v_h^0\|_{m_h}^2+\frac{1}{2}\|u_h^0\|_{a_h}^2+\frac{1}{4}\|(\Pi_{h}^{0,k}u_h^0)^2\|^2-\frac{1}{2}\|u_h^0\|_{m_h}^2.
\end{align}
\end{theorem}
\begin{proof}
Noticing \eqref{202606161004} and \eqref{202606190819}, it holds that 
\begin{align}
	m_h(D_\tau v_h^1,\psi_h)=a_h(D_\tau u_h^1,\psi_h),\quad\forall \psi_h\in V_h^k,\label{202606190828}
\end{align}
Substitute \(\phi_h=D_\tau u_h^1\) into \eqref{202606161007} and \(\psi_h=\gamma \hat{v}_h^1\) into \eqref{202606190828}, then sum the two resulting equations, we arrive at
\begin{align}
\gamma m_h(D_\tau v_h^1,\hat{v}_h^1)+ a_{h}(\hat{u}_h^1,D_\tau u_h^1)+((\Pi_{h}^{0,k}u_h^0)^2\Pi_{h}^{0,k}\hat{u}_h^1,\Pi_{h}^{0,k}D_\tau u_h^1)-m_h(\hat{u}_h^1,D_\tau u_h^1)\leq 0,\label{202606162258}
\end{align}
where we have used the fact that $\|D_\tau u_h^1\|_{m_h}^2\geq0$.

Noting the following elementary results:
\begin{align}
&\gamma m_h(D_\tau v_h^1,\hat{v}_h^1)= \frac{\gamma}{2\tau}(\|v_h^1\|_{m_h}^2-\|v_h^0\|_{m_h}^2),\label{202606162255}\\
&a_{h}(\hat{u}_h^1,D_\tau u_h^1)=\frac{1}{2\tau}(\|u_h^1\|_{a_h}^2-\|u_h^0\|_{a_h}^2),\label{202606162256}\\
&m_h(\hat{u}_h^1,D_\tau u_h^1)=\frac{1}{2\tau}(\|u_h^1\|_{m_h}^2-\|u_h^0\|_{m_h}^2),\label{202606191244}
\end{align}
and 
\begin{align}
((\Pi_{h}^{0,k}u_h^0)^2\Pi_{h}^{0,k}\hat{u}_h^1,\Pi_{h}^{0,k}D_\tau u_h^1)&=\frac{1}{2\tau}((\Pi_{h}^{0,k}u_h^0)^2,(\Pi_{h}^{0,k}u_h^1)^2-(\Pi_{h}^{0,k}u_h^0)^2)\notag\\
&=\frac{1}{2\tau}((\Pi_{h}^{0,k}u_h^0)^2,(\Pi_{h}^{0,k}u_h^1)^2)-\frac{1}{2\tau}\|(\Pi_{h}^{0,k}u_h^0)^2\|^2.\label{202606162257}
\end{align}
Substituting \eqref{202606162255}–\eqref{202606162257} into  \eqref{202606162258} and multiplying both sides by $\frac{\tau}{2}$, we arrive at
\begin{align}
&\quad\frac{\gamma}{4}(\|v_h^1\|_{m_h}^2+\|v_h^{0}\|_{m_h}^2)+\frac{1}{4}(\| u_h^1\|_{a_h}^2+\|u_h^{0}\|_{a_h}^2)\\
&\quad\quad+\frac{1}{4}((\Pi_{h}^{0,k}u_h^1)^2,(\Pi_{h}^{0,k}u_h^{0})^2)-\frac{1}{4}(\|u_h^1\|_{m_h}^2+\|u_h^0\|_{m_h}^2)\\
&\leq \frac{\gamma}{2}\|v_h^0\|_{m_h}^2+\frac{1}{2}\|u_h^0\|_{a_h}^2+\frac{1}{4}\|(\Pi_{h}^{0,k}u_h^0)^2\|^2-\frac{1}{2}\|u_h^0\|_{m_h}^2,
\end{align}
that is to say, $E_h^1\leq E_h^0$. 

Similarly, from \eqref{202606160946}, \eqref{202606161004} and \eqref{202606190819}, it follows that 
\begin{align}
	m_h(D_\tau v_h^n,\psi_h)=a_h(D_\tau u_h^n,\psi_h),\quad\forall \psi_h\in V_h^k,\quad 2\leq n\leq N.\label{202606190941}
\end{align}
After inserting \(\phi_h=D_\tau u_h^n\) into \eqref{202606160945} and \(\psi_h=\gamma \hat{v}_h^n\) into \eqref{202606190941}, we add together the two derived equalities to arrive at
\begin{align}
\gamma m_h(D_\tau v_h^n,\hat{v}_h^n)+ a_{h}(\hat{u}_h^n,D_\tau u_h^n)+((\Pi_{h}^{0,k}u_h^{n-1})^2\Pi_{h}^{0,k}\hat{u}_h^n,\Pi_{h}^{0,k}D_\tau u_h^n)-m_h(\hat{u}_h^n,D_\tau u_h^n)\leq 0.\label{202606162323}
\end{align}
By an argument analogous to that for \eqref{202606162255}–\eqref{202606162257}, we can readily deduce
\begin{align}
&\gamma m_h(D_\tau v_h^n,\hat{v}_h^n)= \frac{\gamma}{4\tau}(\|v_h^n\|_{m_h}^2-\|v_h^{n-2}\|_{m_h}^2),\notag\\
&a_{h}(\hat{u}_h^n,D_\tau u_h^n)=\frac{1}{4\tau}(\|u_h^n\|_{a_h}^2-\|u_h^{n-2}\|_{a_h}^2),\notag\\
&((\Pi_{h}^{0,k}u_h^{n-1})^2\Pi_{h}^{0,k}\hat{u}_h^n,\Pi_{h}^{0,k}D_\tau u_h^n)=\frac{1}{4\tau}((\Pi_{h}^{0,k}u_h^{n-1})^2,(\Pi_{h}^{0,k}u_h^n)^2)\notag\\
&\quad-\frac{1}{4\tau}((\Pi_{h}^{0,k}u_h^{n-1})^2,(\Pi_{h}^{0,k}u_h^{n-2})^2),\notag\\
&m_h(\hat{u}_h^n,D_\tau u_h^n)=\frac{1}{4\tau}(\|u_h^n\|_{m_h}^2-\|u_h^{n-2}\|_{m_h}^2).\notag
\end{align}
Combining the above results, we can immediately get that $$E_h^n\leq E_h^{n-1},\quad 2\leq n\leq N.$$
All this completes the proof.
\end{proof}
\subsection{Boundedness of the virtual element solution}
\begin{theorem}\label{theorem3.3}
Let \(u_h^n\) and \(v_h^n\) be the solutions to the numerical scheme \eqref{202606160945}–\eqref{202606161004}. If \(\tau\leq \frac{1}{3}\), then the following estimate holds:
\begin{align}
\|u_h^n\|\leq C_1\|u_h^{0}\|,\quad 0\leq n\leq N.
\end{align}
\end{theorem}
\begin{proof}
Combining \eqref {202606161004} with \eqref {202606190819}, we deduce that
\begin{align}
	m_h(\hat{v}_h^1,\psi_h)=a_h(\hat{u}_h^1,\psi_h),\quad\forall \psi_h\in V_h^k.\label{202606191347}
\end{align}
Substituting \(\phi_h=\hat{u}_h^1\) into \eqref{202606161007} and \(\psi_h=\gamma \hat{v}_h^1\) into \eqref{202606191347}, adding the two resulting equations yields
\begin{align}
	m_h(D_\tau u_h^1,\hat{u}_h^1)+\|\hat{u}_h^1\|^2_{a_h}+\gamma \|\hat{v}_h\|_{m_h}^2+ ((\Pi_{h}^{0,k}u_h^0)^2\Pi_{h}^{0,k}\hat{u}_h^1,\Pi_{h}^{0,k} \hat{u}_h^1)-\|\hat{u}_h^1\|_{m_h}^2= 0.\label{202606191349}
\end{align}

In view of 
\begin{align}
m_h(D_\tau u_h^1,\hat{u}_h^1)=\frac{1}{2\tau}(\|u_h^1\|_{m_h}^2-\|u_h^0\|_{m_h}^2),
\end{align}
and the nonnegativity of the remaining terms, we have 
\begin{align}
\frac{1}{2\tau}(\|u_h^1\|^2_{m_h}-\|u_h^0\|^2_{m_h})\leq \|\hat{u}_h^1\|^2_{m_h}\leq\left(\frac{\|u_h^1\|_{m_h}+\|u_h^0\|_{m_h}}{2}\right)^2,
\end{align}
where we have also used the triangle inequality.

 Dividing both sides of the above inequality by \(\frac{\|u_h^1\|_{m_h}+\|u_h^0\|_{m_h}}{2}\), we obtain
 \begin{align}
 \frac{1}{\tau}(\|u_h^1\|_{m_h}-\|u_h^0\|_{m_h})\leq \frac{\|u_h^1\|_{m_h}+\|u_h^0\|_{m_h}}{2}.\label{202606191442}
 \end{align}
 When $\tau\leq \frac{2}{3}$, it follows from \eqref{202606191442} that
 \begin{align}
 \|u_h^1\|_{m_h}\leq \frac{1+\frac{\tau}{2}}{1-\frac{\tau}{2}}\|u_h^0\|_{m_h}\leq \left(1+\frac{3\tau}{2}\right)\|u_h^0\|_{m_h}\leq \exp\left(\frac{3\tau}{2}\right)\|u_h^{0}\|_{m_h}\leq\exp\left(3T\right)\|u_h^{0}\|_{m_h},
 \end{align}
 where we have employed the elementary inequality 
 \begin{align}
 &\frac{1+x}{1-x}\leq 1+3x, \quad 0<x\leq\frac{1}{3},\\
 & 1+x\leq \exp(x), \quad x> 0.  
 \end{align}
 
Similarly, in view of \eqref{202606160946}, \eqref{202606161004} and \eqref{202606190819}, it follows that 
 \begin{align}
 	m_h(\hat{v}_h^n,\psi_h)=a_h(\hat{u}_h^n,\psi_h),\quad 2\leq n\leq N,\quad\forall \psi_h\in V_h^k.\label{202606191452}
 \end{align}
 Substituting \(\phi_h=\hat{u}_h^n\) into \eqref{202606160945} and \(\psi_h=\gamma \hat{v}_h^n\) into \eqref{202606191452}, respectively, then summing up these two equations to obtain
 \begin{align}
 m_h(D_\tau u_h^n,\hat{u}_h^n)+\gamma \|\hat{v}_h^n\|^2_{m_h}+ \|\hat{u}_h^n\|_{a_{h}}^2+((\Pi_{h}^{0,k}u_h^{n-1})^2\Pi_{h}^{0,k}\hat{u}_h^n,\Pi_{h}^{0,k} \hat{u}_h^n)-\|\hat{u}_h^n\|_{m_h}^2=0.\label{202606191456}
 \end{align}
 Therefore, it holds that 
 \begin{align}
 \frac{1}{4\tau}(\|u_h^n\|^2_{m_h}-\|u_h^{n-2}\|^2_{m_h})\leq \|\hat{u}_h^n\|^2_{m_h}\leq\left(\frac{\|u_h^n\|_{m_h}+\|u_h^{n-2}\|_{m_h}}{2}\right)^2,
 \end{align}
 where we have used the facts that 
\begin{align}
&m_h(D_\tau u_h^n,\hat{u}_h^n)=\frac{1}{4\tau}(\|u_h^n\|_{m_h}^2-\|u_h^{n-2}\|_{m_h}^2),\\
&\gamma \|\hat{v}_h^n\|^2_{m_h}+ \|\hat{u}_h^n\|_{a_{h}}^2+((\Pi_{h}^{0,k}u_h^{n-1})^2\Pi_{h}^{0,k}\hat{u}_h^n,\Pi_{h}^{0,k} \hat{u}_h^n)\geq 0.
\end{align}
Upon dividing both sides of the above inequality by \(\frac{\|u_h^n\|_{m_h}+\|u_h^{n-2}\|_{m_h}}{2}\), we can deduce the following estimate
\begin{align}
	\frac{1}{2\tau}(\|u_h^n\|_{m_h}-\|u_h^{n-2}\|_{m_h})\leq \frac{\|u_h^n\|_{m_h}+\|u_h^{n-2}\|_{m_h}}{2}.\label{202606191517}
\end{align}
In the case \(\tau \leq \frac{1}{3}\), it follows from \eqref{202606191517} that
\begin{align}
	\|u_h^n\|_{m_h}\leq \frac{1+\tau}{1-\tau}\|u_h^{n-2}\|_{m_h}\leq \left(1+3\tau\right)\|u_h^{n-2}\|_{m_h}\leq \exp\left(3\tau\right)\|u_h^{n-2}\|_{m_h},\quad 2\leq n\leq N,
\end{align}
which by recursion deduces to
\begin{align}
\|u_h^n\|_{m_h}\leq \exp\left(3n\tau\right)\|u_h^{0}\|_{m_h}\leq \exp\left(3T\right)\|u_h^{0}\|_{m_h},\quad 2\leq n\leq N.
\end{align}
Utilizing the norm equivalence associated with the discrete bilinear form \(m_h\), there exists a positive constant \(C_1\) such that
\begin{align}
	\|u_h^n\|\leq C_1\|u_h^{0}\|,\quad 0\leq n\leq N.
\end{align}
The proof is completed. 
\end{proof}
\subsection{Convergence of the virtual element scheme}
\begin{lemma}\label{lemma3.1}
Suppose $u^1,v^1$ are the solutions of the problem \eqref{202606142344}–\eqref{202606142346} and $u_h^1,v_h^1$ are the solutions of the fully discrete scheme \eqref{202606161007}–\eqref{202606161004}. Denote
\begin{align}
\eta_u^n=R_hu^n-u_h^n,\quad \eta_v^n=R_hv^n-v_h^n, \quad 0\leq n\leq N.
\end{align}
Then when $\tau\leq\frac{1}{4}$, there exists a constant $C$ such that
\begin{align}
\|\eta_u^1\|+\tau|\eta_u^1|_1\leq C(\tau^2+h^{k+1}). \label{202606210920}
\end{align}
\end{lemma}
\begin{proof}
By virtue of the discrete scheme \eqref{202606161007} together with the equation \eqref{202606142347} for \(t_{\frac{1}{2}}=\frac{\tau}{2}\), we can easily derive the following error equation:
\begin{align}
&\quad m_h(D_\tau \eta_u^1,\phi_h)+\gamma a_h(\hat{\eta}_v^1,\phi_h)+a_h(\hat{\eta}_u^1,\phi_h)-m_h(\hat{\eta}_u^1,\phi_h)\notag\\
&=m_h(D_\tau R_hu^1,\phi_h)+\gamma a_h(R_h\hat{v}^1,\phi_h)+a_h(R_h\hat{u}^1,\phi_h)-m_h(R_h\hat{u}^1,\phi_h)\notag\\
&\quad-m_h(D_\tau u_h^1,\phi_h)-\gamma a_h(\hat{v}_h^1,\phi_h)-a_h(\hat{u}
_h^1,\phi_h)+m_h(\hat{u}_h^1,\phi_h)\notag\\
&=m_h(D_\tau R_hu^1,\phi_h)+\gamma a_h(R_h\hat{v}^1,\phi_h)+a_h(R_h\hat{u}^1,\phi_h)-m_h(R_h\hat{u}^1,\phi_h)\notag\\
&\quad+((\Pi_{h}^{0,k}u_h^0)^2\Pi_{h}^{0,k}\hat{u}_h^1,\Pi_{h}^{0,k}\phi_h)\notag\\
&=m_h(D_\tau R_hu^1,\phi_h)-\gamma(\Delta \hat{v}^1,\Pi_{h}^{0,k}\phi_h)-(\Delta \hat{u}^1,\Pi_{h}^{0,k}\phi_h)-m_h(R_h\hat{u}^1,\phi_h)\notag\\
&\quad+((\Pi_{h}^{0,k}u_h^0)^2\Pi_{h}^{0,k}\hat{u}_h^1,\Pi_{h}^{0,k}\phi_h)\notag\\
&=m_h(D_\tau R_hu^1,\phi_h)-(D_\tau u^1,\Pi_{h}^{0,k}\phi_h)-((u^0)^2\hat{u}^1,\Pi_{h}^{0,k}\phi_h)+(\hat{u}^1,\Pi_{h}^{0,k}\phi_h)\notag\\
&\quad+(E_u^1,\Pi_{h}^{0,k}\phi_h)-m_h(R_h\hat{u}^1,\phi_h)\notag\\
&\quad+((\Pi_{h}^{0,k}u_h^0)^2\Pi_{h}^{0,k}\hat{u}_h^1,\Pi_{h}^{0,k}\phi_h),\quad \forall\phi_h\in V_h^k,\label{202606191846}
\end{align}
where the temporal discretization error $E^1$ is given by
\begin{align}
 E^1= D_\tau u^1-u(t_{\frac{1}{2}})-\gamma\Delta \hat{v}^1+\gamma\Delta v(t_{\frac{1}{2}})-\Delta \hat{u}^1+\Delta u(t_{\frac{1}{2}})+(u^0)^2\hat{u}^1-u^3(t_{\frac{1}{2}})-\hat{u}^1+u(t_{\frac{1}{2}}).\notag
 \end{align}
Obviously, applying the Taylor expansion theorem, one can easily obtain 
\begin{align}
\|E^1\|\leq C\tau.\label{202606192154}
\end{align}

 Due to the definition of the elliptic projection operator $R_h$ and the discrete equations \eqref{202606161004}-\eqref{202606190819}, we have  
 \begin{align}
 a_h(\hat{\eta}_u^1,\psi_h)&=a_h(R_h\hat{u}^1,\psi_h)-a_h(\hat{u}_h^1,\psi_h)\notag\\
 &=a_h(R_h\hat{u}^1,\psi_h)-m_h(\hat{v}_h^1,\psi_h)\notag\\
 &=-(\Delta\hat{u}^1,\Pi_{h}^{0,k}\psi_h)-m_h(\hat{v}_h^1,\psi_h)\notag\\
 &=(\hat{v}^1,\Pi_{h}^{0,k}\psi_h)-m_h(\hat{v}_h^1,\psi_h).\label{202606191838}
 \end{align}
 
Setting \(\phi_h=\hat{\eta}_u^1\) in \eqref{202606191846} and \(\psi_h=\gamma\hat{\eta}_v^1\) in the corresponding discrete relation \eqref{202606191838}, respectively, and subtracting the two resulting identities, we obtain
 \begin{align}
 m_h(D_\tau \eta_u^1,\hat{\eta}_u^1)+\|\hat{\eta}_u^1\|^2_{a_h}-\|\hat{\eta}_u^1\|^2_{m_h}=A_1+A_2+A_3+A_4+A_5,\notag
 \end{align}
 where 
 \begin{align}
 &A_1 = m_h(D_\tau R_hu^1,\hat{\eta}_u^1)-(D_\tau u^1,\Pi_{h}^{0,k}\hat{\eta}_u^1),\notag\\
 &A_2=(\hat{u}^1,\Pi_{h}^{0,k}\hat{\eta}_u^1)-m_h(R_h\hat{u}^1,\hat{\eta}_u^1),\notag\\
 &A_3 = -((u^0)^2\hat{u}^1,\Pi_{h}^{0,k}\hat{\eta}_u^1)+((\Pi_{h}^{0,k}u_h^0)^2\Pi_{h}^{0,k}\hat{u}_h^1,\Pi_{h}^{0,k}\hat{\eta}_u^1),\notag\\
 &A_4 = (E^1,\Pi_{h}^{0,k}\hat{\eta}_u^1),\notag\\
 &A_5 = \gamma m_h(\hat{v}_h^1,\hat{\eta}_v^1)-\gamma(\hat{v}^1,\Pi_{h}^{0,k}\hat{\eta}_v^1).\notag 
 \end{align}
 
 Using the simple relation
 \begin{align}
  m_h(D_\tau \eta_u^1,\hat{\eta}_u^1)=\frac{1}{2\tau}(\|\eta_u^1\|_{m_h}^2-\|\eta_u^0\|_{m_h}^2),\quad \eta_u^0=0,\notag
 \end{align}
 we have 
 \begin{align}
 \|\eta_u^1\|_{m_h}^2+2\tau\|\eta_u^1\|^2_{a_h}-2\tau\|\eta_u^1\|_{m_h}^2\leq 2\tau A_1+2\tau A_2+2\tau A_3+2\tau A_4+2\tau A_5.\label{202606200826}
 \end{align}
 
 Next, we aim to bound every term on the right-hand side of the foregoing equation.
 For the first term, using the Cauchy-Schwarz inequality and Young inequality, it holds that  
 \begin{align}
 2\tau A_1&=2\tau m_h(D_\tau R_hu^1,\hat{\eta}_u^1)-2\tau(D_\tau u^1,\Pi_{h}^{0,k}\hat{\eta}_u^1)\notag\\
 &=2\tau m_h(D_\tau R_hu^1-D_\tau \Pi_{h}^{0,k}u^1,\hat{\eta}_u^1)+2\tau m_h(D_\tau \Pi_{h}^{0,k}u^1,\hat{\eta}_u^1)-2\tau(D_\tau u^1,\Pi_{h}^{0,k}\hat{\eta}_u^1)\notag\\
 &=2\tau m_h(D_\tau R_hu^1-D_\tau \Pi_{h}^{0,k}u^1,\hat{\eta}_u^1)+2\tau(D_\tau \Pi_{h}^{0,k}u^1,\hat{\eta}_u^1)-2\tau(D_\tau u^1,\Pi_{h}^{0,k}\hat{\eta}_u^1)\notag\\
 &=2\tau m_h(D_\tau R_hu^1-D_\tau \Pi_{h}^{0,k}u^1,\hat{\eta}_u^1)\notag\\
 &\leq 2\tau \|D_\tau R_hu^1-D_\tau \Pi_{h}^{0,k}u^1\|_{m_h}\|\hat{\eta}_u^1\|_{m_h}\notag\\
 &\leq C\tau\|D_\tau R_hu^1-D_\tau \Pi_{h}^{0,k}u^1\|\|\eta_u^1\|_{m_h}\notag\\
 &\leq C \tau(\|D_\tau R_hu^1-D_\tau u^1\|+\|D_\tau u^1-D_\tau \Pi_{h}^{0,k}u^1\|)\|\eta_u^1\|_{m_h}\notag\\
 &\leq C (\|R_{h}(u^1-u^0)-(u^1-u^0)\|+\|\Pi_h^k(u^1-u^0)-(u^1-u^0)\|)\|\eta_u^1\|_{m_h}\notag\\
 &\leq \frac{1}{16}\|\eta_u^1\|_{m_h}^2+Ch^{2k+2},\label{202606200825}
 \end{align}
 where we have used the consisitency of the disrete mass operator $m_h(\cdot,\cdot)$, i.e., 
 \begin{align}
 m_h(D_\tau \Pi_h^ku^1,\hat{\eta}_u^1)= (D_\tau \Pi_h^ku^1,\hat{\eta}_u^1), \notag
 \end{align}
 the property of the $L^2$-projection
 \begin{align}
(D_\tau \Pi_h^ku^1,\hat{\eta}_u^1)=(D_\tau u^1,\Pi_{h}^{0,k}\hat{\eta}_u^1),\notag
 \end{align}
 and the approximation properties of the projection operator $R_h$ and $\Pi_h^k$
 \begin{align}
 &\|R_{h}(u^1-u^0)-(u^1-u^0)\|\leq Ch^{k+1}\|u^1-u^0\|_{k+1}\notag\\
 &\quad\leq 
 Ch^{k+1}\left\|\int_{t_0}^{t_1}u_t\text{d}s\right\|_{k+1}\leq Ch^{k+1} \int_{t_0}^{t_1}\|u_t\|_{k+1}\text{d}s,\notag\\
 &\|\Pi_h^k(u^1-u^0)-(u^1-u^0)\|\leq Ch^{k+1} \int_{t_0}^{t_1}\|u_t\|_{k+1}\text{d}s.\notag 
 \end{align} 
 
 For the second term, we adopt arguments analogous to those used for the first term. In fact, we have 
 \begin{align}
 2\tau A_2= (\hat{u}^1,\Pi_{h}^{0,k}\hat{\eta}_u^1)-m_h(R_h\hat{u}^1,\hat{\eta}_u^1)\leq \frac{1}{16}\|\eta_u^1\|_{m_h}^2+Ch^{2k+2}.\label{202606192145}
 \end{align}
 
 We now turn to the third term, namely the nonlinear term. For this, by virtue of the Holder inequality, it holds that  
\begin{align}
A_3 &=-((u^0)^2\hat{u}^1,\Pi_{h}^{0,k}\hat{\eta}_u^1)+((\Pi_{h}^{0,k}u_h^0)^2\Pi_{h}^{0,k}\hat{u}_h^1,\Pi_{h}^{0,k}\hat{\eta}_u^1)\notag\\
&=-((u^0)^2\hat{u}^1,\Pi_{h}^{0,k}\hat{\eta}_u^1)+((\Pi_{h}^{0,k}u^0)^2\hat{u}^1,\Pi_{h}^{0,k}\hat{\eta}_u^1)-((\Pi_{h}^{0,k}u^0)^2\hat{u}^1,\Pi_{h}^{0,k}\hat{\eta}_u^1)\notag\\
&\quad+((\Pi_{h}^{0,k}u^0)^2\Pi_{h}^{0,k}\hat{u}^1,\Pi_{h}^{0,k}\hat{\eta}_u^1)-((\Pi_{h}^{0,k}u^0)^2\Pi_{h}^{0,k}\hat{u}^1,\Pi_{h}^{0,k}\hat{\eta}_u^1)\notag\\
&\quad+((\Pi_{h}^{0,k}u^0_h)^2\Pi_{h}^{0,k}\hat{u}^1,\Pi_{h}^{0,k}\hat{\eta}_u^1)-((\Pi_{h}^{0,k}u_h^0)^2\Pi_{h}^{0,k}\hat{u}^1,\Pi_{h}^{0,k}\hat{\eta}_u^1)+((\Pi_{h}^{0,k}u_h^0)^2\Pi_{h}^{0,k}\hat{u}_h^1,\Pi_{h}^{0,k}\hat{\eta}_u^1)\notag\\
&\leq \frac{1}{16}\|\eta_u^1\|_{m_h}^2+Ch^{2k+2},
 \end{align}
where the following \(L^\infty\)-norm boundedness has been repeatedly utilized
\begin{align}
\|u_h^0\|_{0,\infty}=\|R_hu^0\|_{0,\infty}\leq C,\quad \|\Pi_{h}^{0,k}u_h^0\|_{0,\infty}\leq C\|u_h^0\|_{0,\infty}\leq C,\notag
\end{align}
and we further exploit the continuity property of the \(L^2\) projection operator
\begin{align}
\|\Pi_{h}^{0,k}\hat{u}_h^1\|\leq C\|\hat{u}_h^1\|.\notag 
\end{align}

Combining \eqref{202606192154} with the Cauchy–Schwarz inequality yields
\begin{align}
2\tau A_4&= 2\tau (E^1,\Pi_{h}^{0,k}\hat{\eta}_u^1)\notag\\
&\leq C\tau\|E^1\|\|\Pi_{h}^{0,k}\hat{\eta}_u^1\|\notag\\
&\leq \frac{1}{16}\|\eta_u^1\|_{m_h}^2+C\tau^{4}.
\end{align}

Lastly, we proceed to bound the fifth term. In fact, we have 
\begin{align}
A_5&=\gamma m_h(\hat{v}_h^1,\hat{\eta}_v^1)-\gamma(\hat{v}^1,\Pi_{h}^{0,k}\hat{\eta}_v^1)\notag\\
&=\gamma m_h(R_h\hat{v}^1,\hat{\eta}_v^1)-\gamma \|\hat{\eta}_v^1\|_{m_h}^2-\gamma(\hat{v}^1,\Pi_{h}^{0,k}\hat{\eta}_v^1)\notag\\
&\leq Ch^{2k+2}.\label{202606200824}
\end{align}

Substituting \eqref{202606200825}–\eqref{202606200824} into \eqref{202606200826}, we obtain
\begin{align}
(1-2\tau)\|\eta_u^1\|_{m_h}^2+2\tau\|\eta_u^1\|^2_{a_h}\leq \frac{1}{4}\|\eta_u^1\|_{m_h}^2+Ch^{2k+2}+C\tau^4.
\end{align} 
Therefore, when $\tau\leq\frac{1}{4}$, we have 
\begin{align}
\frac{1}{4}\|\eta_u^1\|_{m_h}^2+2\tau\|\eta_u^1\|^2_{a_h}\leq Ch^{2k+2}+C\tau^4.\label{202606210924}
\end{align}
Combined with the norm equivalence properties of $m_h$ and $a_h$, it holds that 
\begin{align}
\|u_h^1\|+\tau|u_h^1|_1\leq C(\tau^2+h^{k+1}). 
\end{align}
This completes the proof. 
\end{proof}
\begin{lemma}\label{lemma3.2}
Let \(u^n, v^n\) solve the continuous problem \eqref{202606142344}–\eqref{202606142346}, and let \(u_h^n, v_h^n\) denotes the solution to the fully discrete scheme \eqref{202606160945}-\eqref{202606161004}. 
Then we can find two constants $\tau_0$ and $h_0$ such that when $\tau\leq\tau_0$ and $h\leq h_0$ 
\begin{align}
\|\eta_h^n\|+\tau|\eta_h^n|_1\leq C(\tau^2+h^{k+1}),\quad 0\leq n\leq N.\label{202606200847}
\end{align}
\end{lemma}
\begin{proof}
We adopt the method of mathematical induction to verify that formula \ref{202606200847} holds. It follows from Lemma \ref{lemma3.1} and the definition of \(u_h^0\) that the inequality holds for \(n=0,1\). Assume the conclusion holds for \(k\leq n-1\),  i.e., 
\begin{align}
\|\eta_h^k\|+\tau|\eta_h^k|_1\leq C(\tau^2+h^{k+1}),\quad 0\leq k\leq n-1
.\label{202606200854}
\end{align}
When $\tau\leq h$ and $h$ is sufficiently small, making use of the inverse inequality for virtual element function, we have 
\begin{align}
|\eta_u^k|_1\leq Ch^{-1}\|\eta_u^k\|\leq Ch\leq 1, \quad 0\leq k\leq n-1.
\end{align}
When \(\tau> h\), using \eqref{202606200854}, we have
\begin{align}
|\eta_u^k|_1\leq C\tau\leq1,\quad 0\leq k\leq n-1.
\end{align}
From the above estimates, it follows that
\begin{align}
|\eta_u^k|_1\leq1,\quad 0\leq k\leq n-1.\label{202606201538}
\end{align}

Next, we aim to prove that the conclusion holds for $n$. To this end, from the fully discrete numerical scheme \eqref{202606160945}–\eqref{202606160946}, we have
\begin{align}
&\quad m_h(D_\tau \eta_u^n,\phi_h)+\gamma a_h(\hat{\eta}_v^n,\phi_h)+a_h(\hat{\eta}_u^n,\phi_h)-m_h(\hat{\eta}_u^n,\phi_h)\notag\\
&=m_h(D_\tau R_hu^n,\phi_h)+\gamma a_h(R_h\hat{v}^n,\phi_h)+a_h(R_h\hat{u}^n,\phi_h)-m_h(R_h\hat{u}^n,\phi_h)\notag\\
&\quad-m_h(D_\tau u_h^n,\phi_h)-\gamma a_h(\hat{v}_h^n,\phi_h)-a_h(\hat{u}
_h^n,\phi_h)+m_h(\hat{u}_h^n,\phi_h)\notag\\
&=m_h(D_\tau R_hu^n,\phi_h)+\gamma a_h(R_h\hat{v}^n,\phi_h)+a_h(R_h\hat{u}^n,\phi_h)-m_h(R_h\hat{u}^n,\phi_h)\notag\\
&\quad+((\Pi_{h}^{0,k}u_h^{n-1})^2\Pi_{h}^{0,k}\hat{u}_h^n,\Pi_{h}^{0,k}\phi_h)\notag\\
&=m_h(D_\tau R_hu^n,\phi_h)-\gamma(\Delta \hat{v}^n,\Pi_{h}^{0,k}\phi_h)-(\Delta \hat{u}^n,\Pi_{h}^{0,k}\phi_h)-m_h(R_h\hat{u}^n,\phi_h)\notag\\
&\quad+((\Pi_{h}^{0,k}u_h^{n-1})^2\Pi_{h}^{0,k}\hat{u}_h^n,\Pi_{h}^{0,k}\phi_h)\notag\\
&=m_h(D_\tau R_hu^n,\phi_h)-(D_\tau u^n,\Pi_{h}^{0,k}\phi_h)-((u^{n-1})^2\hat{u}^n,\Pi_{h}^{0,k}\phi_h)+(\hat{u}^n,\Pi_{h}^{0,k}\phi_h)\notag\\
&\quad+(E^n,\Pi_{h}^{0,k}\phi_h)-m_h(R_h\hat{u}^n,\phi_h)\notag\\
&\quad+((\Pi_{h}^{0,k}u_h^{n-1})^2\Pi_{h}^{0,k}\hat{u}_h^n,\Pi_{h}^{0,k}\phi_h),\quad \forall\phi_h\in V_h^k,\quad 2\leq n\leq N,\label{202606200952}
\end{align}
where the time discretization error $E^n$ at the $(n-1)$-th time level is defined as  
\begin{align}
E^n&= D_\tau u^{n}-u(t_{n-1})-\gamma\Delta \hat{v}^n+\gamma\Delta v(t_{n-1})\notag\\
&\quad-\Delta \hat{u}^n+\Delta u(t_{n-1})+(u^{n-1})^2\hat{u}^n-u^3(t_{n-1})-\hat{u}^n+u(t_{n-1}).\notag
\end{align}

Employing a similar argument as in \eqref{202606191838}, we arrive at 
\begin{align}
a_h(\hat{\eta}_u^n,\psi_h)&=a_h(R_h\hat{u}^n,\psi_h)-a_h(\hat{u}_h^n,\psi_h)\notag\\
&=a_h(R_h\hat{u}^n,\psi_h)-m_h(\hat{v}_h^n,\psi_h)\notag\\
&=-(\Delta\hat{u}^n,\Pi_{h}^{0,k}\psi_h)-m_h(\hat{v}_h^n,\psi_h)\notag\\
&=(\hat{v}^n,\Pi_{h}^{0,k}\psi_h)-m_h(\hat{v}_h^n,\psi_h)\label{202606201010}. 
\end{align}

Substituting \(\phi_h=\hat{\eta}_u^n\) into \eqref{202606200952} and \(\psi_h=\gamma\hat{\eta}_v^n\) into the corresponding discrete equation \eqref{202606201010}, and subtracting the two resulting equalities, we obtain
 \begin{align}
m_h(D_\tau \eta_u^n,\hat{\eta}_u^n)+\|\hat{\eta}_u^n\|^2_{a_h}-\|\hat{\eta}_u^n\|^2_{m_h}=B_1+B_2+B_3+B_4+B_5,\label{202606201014}
\end{align}
where 
\begin{align}
&B_1 = m_h(D_\tau R_hu^n,\hat{\eta}_u^n)-(D_\tau u^n,\Pi_{h}^{0,k}\hat{\eta}_u^n),\notag\\
&B_2=(\hat{u}^n,\Pi_{h}^{0,k}\hat{\eta}_u^n)-m_h(R_h\hat{u}^n,\hat{\eta}_u^n),\notag\\
&B_3 = -((u^{n-1})^2\hat{u}^n,\Pi_{h}^{0,k}\hat{\eta}_u^n)+((\Pi_{h}^{0,k}u_h^{n-1})^2\Pi_{h}^{0,k}\hat{u}_h^n,\Pi_{h}^{0,k}\hat{\eta}_u^n),\notag\\
&B_4 = (E^n,\Pi_{h}^{0,k}\hat{\eta}_u^n),\notag\\
&B_5 = \gamma m_h(\hat{v}_h^n,\hat{\eta}_v^n)-\gamma(\hat{v}^n,\Pi_{h}^{0,k}\hat{\eta}_v^n).\notag 
\end{align}

By following the estimates of \(A_1, A_2, A_4, A_5\) presented in Lemma \ref{lemma3.1}, we can readily derive the following bounds for \(B_1, B_2, B_4, B_5\), i.e., 
\begin{align}
B_1+B_2+B_4+B_5\leq C(\|\eta_h^n\|^2_{m_h}+\|\eta_h^{n-2}\|^2_{m_h}+\tau^4+h^{2k+2}).\notag
\end{align}
For the estimate of $B_3$, it holds that 
\begin{align}
B_3 &=-((u^{n-1})^2\hat{u}^n,\Pi_{h}^{0,k}\hat{\eta}_u^n)+((\Pi_{h}^{0,k}u_h^{n-1})^2\Pi_{h}^{0,k}\hat{u}_h^n,\Pi_{h}^{0,k}\hat{\eta}_u^n),\notag\\
&=-((u^{n-1})^2\hat{u}^n,\Pi_{h}^{0,k}\hat{\eta}_u^n)+((\Pi_{h}^{0,k}u^{n-1})u^{n-1}\hat{u}^n,\Pi_{h}^{0,k}\hat{\eta}_u^n)-((\Pi_{h}^{0,k}u^{n-1})u^{n-1}\hat{u}^n,\Pi_{h}^{0,k}\hat{\eta}_u^n)\notag\\
&\quad+((\Pi_{h}^{0,k}u^{n-1})^2\hat{u}^n,\Pi_{h}^{0,k}\hat{\eta}_u^n)-((\Pi_{h}^{0,k}u^{n-1})^2\hat{u}^1,\Pi_{h}^{0,k}\hat{\eta}_u^n)\notag\\
&\quad+((\Pi_{h}^{0,k}u^{n-1})^2\Pi_{h}^{0,k}\hat{u}^n,\Pi_{h}^{0,k}\hat{\eta}_u^n)-((\Pi_{h}^{0,k}u^{n-1})^2\Pi_{h}^{0,k}\hat{u}^n,\Pi_{h}^{0,k}\hat{\eta}_u^n)\notag\\
&\quad+((\Pi_{h}^{0,k}u_h^{n-1})^2\Pi_{h}^{0,k}\hat{u}_h^n,\Pi_{h}^{0,k}\hat{\eta}_u^n)=:B_{31}+B_{32}+B_{33}+B_{34}.\label{202606201035}
\end{align}

Subsequently, we estimate each term in identity \eqref{202606201035}. In fact, by virtue of the Cauchy–Schwarz inequality and Young inequality, we arrive at
\begin{align}
B_{31} &= -((u^{n-1})^2\hat{u}^n,\Pi_{h}^{0,k}\hat{\eta}_u^n)+((\Pi_{h}^{0,k}u^{n-1})u^{n-1}\hat{u}^n,\Pi_{h}^{0,k}\hat{\eta}_u^n)\notag\\
& \leq \|u^{n-1}\hat{u}^n\|_{0,\infty}\|u^{n-1}-\Pi_{h}^{0,k}u^{n-1}\|\|\Pi_{h}^{0,k}\hat{\eta}_u^n\|\notag\\
&\leq C\|\eta_u^n\|^2_{m_h}+C\|\eta_u^{n-2}\|^2_{m_h}+Ch^{2k+2}.
\notag \end{align}
 
Following the similar procedure yields 
\begin{align}
B_{32}+B_{33}\leq C\|\eta_u^n\|^2_{m_h}+C\|\eta_u^{n-2}\|^2_{m_h}+Ch^{2k+2}.\notag
\end{align}

Employing the following facts,
\begin{align}
\hat{u}^n_h=R_h\hat{u}^n-\hat{\eta}_u^n, \quad \hat{u}^n= \hat{u}^n-R_h\hat{u}^n+R_h\hat{u}^n=:\hat{\xi}^n+R_h\hat{u}^n,\notag
\end{align}
and the Gagliardo–Nirenberg interpolation inequality
\begin{align}
\|\eta_u^{n-1}\|_{0,4}^2\leq C\|\eta_u^{n-1}\||\eta_u^{n-1}|_1,\notag
\end{align}
we arrive at 
\begin{align}
B_{34}&=-((\Pi_{h}^{0,k}u^{n-1})^2\Pi_{h}^{0,k}\hat{u}^n,\Pi_{h}^{0,k}\hat{\eta}_u^n)+((\Pi_{h}^{0,k}u_h^{n-1})^2\Pi_{h}^{0,k}\hat{u}_h^n,\Pi_{h}^{0,k}\hat{\eta}_u^n)\notag\\
&=-((\Pi_{h}^{0,k}u^{n-1})^2\Pi_{h}^{0,k}\hat{u}^n,\Pi_{h}^{0,k}\hat{\eta}_u^n)+((\Pi_{h}^{0,k}u_h^{n-1})^2(\Pi_{h}^{0,k}R_h\hat{u}^n),\Pi_{h}^{0,k}\hat{\eta}_u^n)\notag\\
&\quad-((\Pi_{h}^{0,k}u_h^{n-1})^2\Pi_{h}^{0,k}\hat{\eta}_u^n,\Pi_{h}^{0,k}\hat{\eta}_u^n)\notag\\
&\leq  -((\Pi_{h}^{0,k}u^{n-1})^2\Pi_{h}^{0,k}\hat{\xi}^n,\Pi_{h}^{0,k}\hat{\eta}_u^n)-((\Pi_{h}^{0,k}u^{n-1})^2(\Pi_{h}^{0,k}R_h\hat{u}^n),\Pi_{h}^{0,k}\hat{\eta}_u^n)\notag\\
&\quad+((\Pi_{h}^{0,k}u_h^{n-1})^2(\Pi_{h}^{0,k}R_h\hat{u}^n),\Pi_{h}^{0,k}\hat{\eta}_u^n)\notag\\
&\leq -((\Pi_{h}^{0,k}u^{n-1})^2\Pi_{h}^{0,k}\hat{\xi}^n,\Pi_{h}^{0,k}\hat{\eta}_u^n)\notag\\
&\quad-((\Pi_{h}^{0,k}\xi^{n-1}+\Pi_{h}^{0,k}\eta_u^{n-1})(\Pi_{h}^{0,k}R_hu^{n-1}-\Pi_{h}^{0,k}\eta_u^{n-1}+\Pi_{h}^{0,k}u^{n-1})(\Pi_{h}^{0,k}R_h\hat{u}^n),\Pi_{h}^{0,k}\hat{\eta}_u^n)\notag\\
&\leq C\|\Pi_{h}^{0,k}\eta_u^{n-1}\|_{0,4}^2\|\Pi_{h}^{0,k}\hat{\eta}_u^n\|+C\|\Pi_{h}^{0,k}\eta_u^n\|^2+Ch^{2k+2}+C\|\Pi_{h}^{0,k}\eta_u^{n-1}\|^2+C\|\Pi_{h}^{0,k}\eta_u^{n-2}\|^2\label{202606201513}\notag\\
&\leq C\|\eta_u^{n-1}\|_{0,4}^2\|\Pi_{h}^{0,k}\hat{\eta}_u^n\|+C\|\Pi_{h}^{0,k}\eta_u^n\|^2+Ch^{2k+2}+C\|\Pi_{h}^{0,k}\eta_u^{n-1}\|^2+C\|\Pi_{h}^{0,k}\eta_u^{n-2}\|^2\notag\\
&\leq C\|\eta_u^{n-1}\||\eta_u^{n-1}|_1\|\Pi_{h}^{0,k}\hat{\eta}_u^n\|+C\|\Pi_{h}^{0,k}\eta_u^n\|^2+Ch^{2k+2}+C\|\Pi_{h}^{0,k}\eta_u^{n-1}\|^2+C\|\Pi_{h}^{0,k}\eta_u^{n-2}\|^2\notag\\
&\leq C\|\eta_u^{n-1}\|\|\Pi_{h}^{0,k}\hat{\eta}_u^n\|+C\|\Pi_{h}^{0,k}\eta_u^n\|^2+Ch^{2k+2}+C\|\Pi_{h}^{0,k}\eta_u^{n-1}\|^2+C\|\Pi_{h}^{0,k}\eta_u^{n-2}\|^2\notag\\
&\leq C(\|\eta_u^n\|_{m_h}^2+\|\eta_u^{n-1}\|_{m_h}^2+\|\eta_u^{n-2}\|_{m_h}^2)+Ch^{2k+2},\quad 2\leq n\leq N\notag,
\end{align}
where we have also used \eqref{202606201538}. From the foregoing bounds, it follows that
\begin{align}
B_3\leq C(\|\eta_u^n\|_{m_h}^2+\|\eta_u^{n-1}\|_{m_h}^2+\|\eta_u^{n-2}\|_{m_h}^2)+Ch^{2k+2},\quad 2\leq n\leq N.\notag
\end{align}

Substituting the estimates of \(B_1,B_2,B_3,B_4,B_5\) into \eqref{202606201014}, taking \(\tau\) sufficiently small, summing over both sides, and multiplying the resulting equation by \(\tau\), we obtain
\begin{align}
\|\eta_u^j\|_{m_h}^2+\|\eta_u^{j-1}\|_{m_h}^2-\|\eta_u^1\|_{m_h}^2+\tau\sum\limits_{k=2}^{j}\|\hat{\eta}_u^k\|^2_{a_h}\leq  \tau\sum\limits_{k=1}^{j}\|\eta_u^k\|_{m_h}^2+C(\tau^4+h^{2k+2}),\quad 2\leq j\leq n,\notag
\end{align}
where we have utilized 
\begin{align}
m_h(D_\tau \eta_u^j,\hat{\eta}_u^j)=\frac{1}{4\tau}(\| \eta_u^j\|_{m_h}^2-\|\eta_u^{j-2}\|_{m_h}^2),\quad \eta_u^0=0.\notag
\end{align}

Adding the above inequality to \eqref{202606210924}, we obtain
\begin{align}
\|\eta_u^j\|_{m_h}^2+\tau\sum\limits_{k=1}^{j}\|\hat{\eta}_u^j\|^2_{a_h}\leq \tau  \sum\limits_{k=1}^{j}\|\eta_u^k\|_{m_h}^2+C(\tau^4+h^{2k+2}),\quad 2\leq j\leq n. \notag 
\end{align}
Obviously, it is clear that the above formula also holds for \(j=1\), i.e., 
\begin{align}
	\|\eta_u^j\|_{m_h}^2+\tau\sum\limits_{k=1}^{j}\|\hat{\eta}_u^k\|^2_{a_h}\leq \tau \sum\limits_{k=1}^{j}\|\eta_u^k\|_{m_h}^2+C(\tau^4+h^{2k+2}),\quad 1\leq j\leq n. \notag
\end{align}
By using the Gronwall inequality given in Lemma \ref{lemma3.1}, we obtain 
\begin{align}
\|\eta_u^n\|_{m_h}^2+\tau\sum\limits_{k=1}^{n}\|\hat{\eta}_u^n\|^2_{a_h}\leq C(\tau^4+h^{2k+2}),\quad 2\leq n\leq N.\notag 
\end{align}

Finally, using the shift inequality in Lemma \ref{lemma2.1}, we have
\begin{align}
\tau \|\eta_u^n\|_{a_h}&\leq C\tau\sum\limits_{k=1}^{n}\|\hat{\eta}_u^k\|_{a_h}\notag\\
&\leq C\sum\limits_{k=1}^{n}\tau^{\frac{1}{2}}(\tau^{\frac{1}{2}}\|\hat{\eta}_u^k\|_{a_h})\notag\\
&\leq \sqrt{n\tau}\sqrt{\sum\limits_{k=1}^{n}\tau\|\hat{\eta}_u^k\|_{a_h}^2}\notag\\
&\leq C(\tau^2+h^{k+1}).
\end{align}
Therefore, the conclusion \eqref{202606200847} is true.  
\end{proof}
\begin{theorem}
Suppose \(u^n, v^n\) are the solutions to the continuous problem \eqref{202606142344}–\eqref{202606142346}, while \(u_h^n, v_h^n\) denote the approximate solutions to the fully discrete scheme \eqref{202606160945}–\eqref{202606161004}.
Then, for \(0<\tau\leq\frac{1}{3}\) and \(h>0\), we have
\begin{align}
\|u^n-u_h^n\|\leq C(\tau^{4}+h^{k+1}),\quad 0\leq n\leq N.
\end{align}
\end{theorem}
\begin{proof}
By employing the triangle inequality and Lemma \ref{lemma3.2}, when $\tau\leq\tau_0$ and $h\leq h_0$, it holds that 
\begin{align}
\|u-u_h^n\|\leq \|\xi^n\|+\|\eta_u^n\|\leq C(\tau^{2}+h^{k+1}),\quad 0\leq n \leq N.
\end{align}	
Under the three scenarios: \(\tau>\tau_0, h\leq h_0\), \(\tau>\tau_0, h>h_0\) and \(\tau\leq \tau_0, h>h_0\), it is straightforward to verify the existence of a positive constant $C_2$ such that
\begin{align}
	\tau^2+h^{k+1}\geq C_2.
\end{align}
Therefore, with the help of the Theorem \ref{theorem3.3}, we have 
\begin{align}
\|u^n-u_h^n\|\leq \|u^n\|+\|u_h^n\|&\leq \|u^n\|+C_1\|u_h^0\|\notag\\
&\leq \frac{1}{C_2}(\|u^n\|+C_1\|u_h^0\|)(\tau^2+h^{k+1})\notag\\
&\leq C(\tau^2+h^{k+1}).
\end{align}
The proof is completed. 
\end{proof}
\section{Numerical examples}\label{section4}

In this section, two numerical examples are presented to verify the correctness of the foregoing theoretical analysis. All numerical results are implemented using MATLAB R2025a. Since the virtual element basis functions cannot be expressed explicitly, we adopt the following numerical error to verify the theoretical convergence orders:
\begin{align}
	L^2\text{-error}=\left(\sum\limits_{K\in \mathcal{T}_h}\|u^N-\Pi_{K}^{0,k}u_h^N\|_{K}^2\right)^{1/2},
	\quad H^1\text{-error}=\left(\sum\limits_{K\in \mathcal{T}_h}|u^N-\Pi_{K}^{1,k}u_h^N|^2_{1,K}\right)^{1/2}.\notag
\end{align}
\begin{example}\label{example1}
In the first example, we choose $\gamma=1$ in \eqref{202606142344}-\eqref{202606142346}, the exact solution is taken as 
\begin{align}
u(x,y,t)= e^{-t}\sin(\pi x)\sin(\pi y),\notag
\end{align}
and the right-hand side function is derived from the above exact solution. 

This example aims to verify the convergence accuracy of the numerical solutions using a constructed artificial exact solution. Numerical simulations are carried out on two types of polygonal meshes, namely non-convex meshes and Voronoi meshes (see Figure \ref{figure1}).
To test the spatial convergence accuracy, we set \(T=1\mathrm{e}{-4}\) and \(\tau=1\mathrm{e}{-6}\), and adopt a series of spatial refinement parameters. Numerical results displayed in Tables \ref{table1}–\ref{table4} indicate that the proposed numerical scheme achieves the optimal convergence orders in the \(L^2\)-norm and \(H^1\)-norm. To examine the temporal convergence accuracy, we set \(k=2\) and \(\tau=h^2\). The corresponding numerical results are presented in Table \ref{table5}, which demonstrate that the temporal convergence order reaches the expected second-order accuracy. The above numerical results fully demonstrate the effectiveness of the proposed numerical algorithm and validate the correctness of the previous theoretical analysis.
\begin{figure}[!h]
	\centering
	\subfigure[Mesh with non-convex elements]{\begin{minipage}[t]{0.48\linewidth}
			\includegraphics[width=1\linewidth]{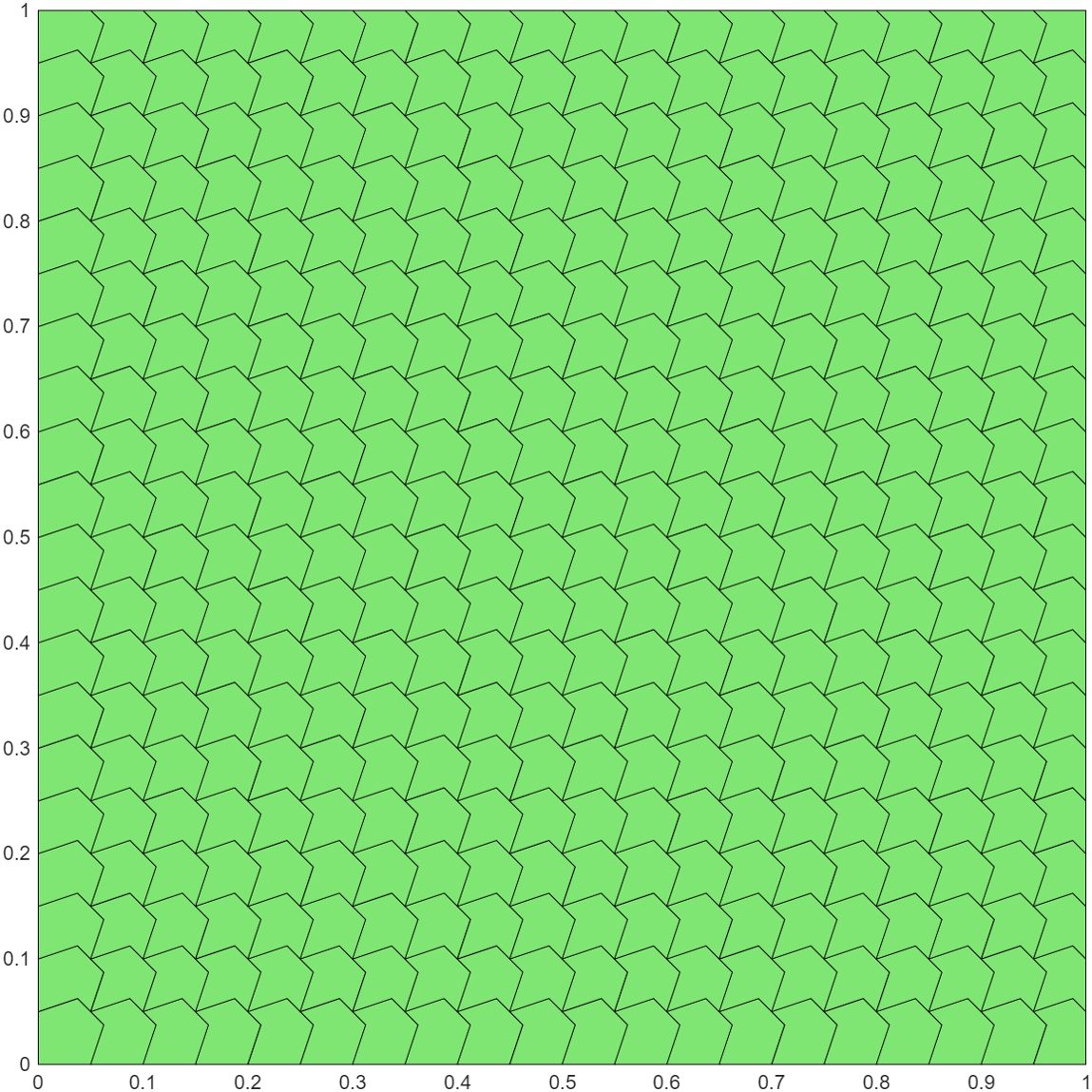}
		\end{minipage}
	}
	\subfigure[Mesh with Voronoi elements ]{\begin{minipage}[t]{0.48\linewidth}
			\includegraphics[width=1\linewidth]{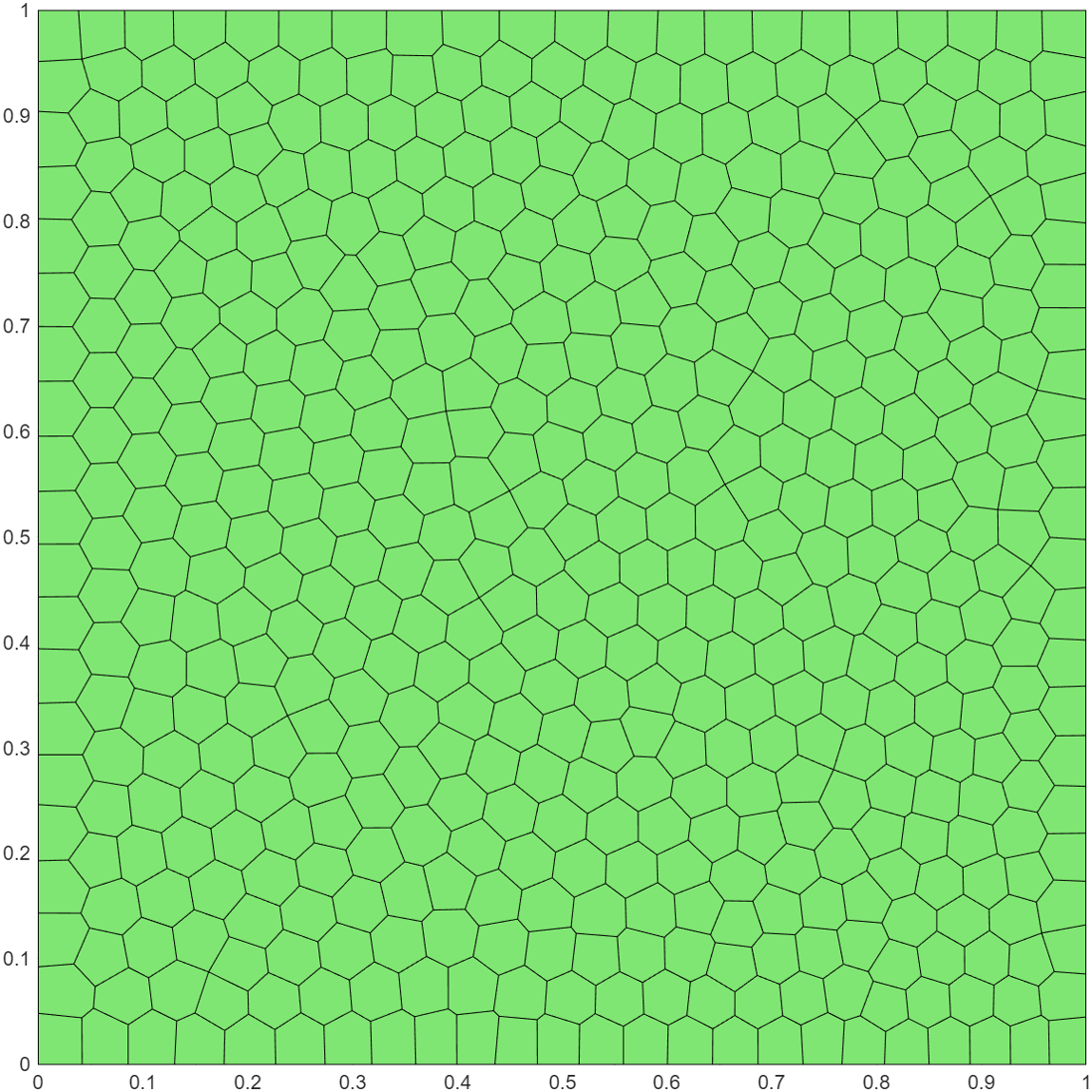}
		\end{minipage}
	}
	\caption{Polygonal meshes setup corresponding to Example \ref{example1}.}
	\label{figure1}
\end{figure}
\end{example}

\begin{table}
	\centering
	\caption{Convergence order in spatial direction for Example \ref{example1} on Voronoi meshes with \(k=1\).}
	\begin{tabular}{ccccc}
		\toprule 
		$h$ & $L^2$-error & $\text{Order}$ & $H^1$-error & $\text{Order}$\\
		\midrule
			1/$\sqrt{32}$&2.3948e-02	&  &	 5.1320e-01	 &\\
		1/$\sqrt{64}$&1.1851e-02	& 2.0299	& 3.6358e-01&	 0.9945\\
			1/$\sqrt{128}$&5.8750e-03	& 2.0246	& 2.4817e-01&	 1.1019\\
	1/$\sqrt{256}$	&2.9375e-03	 &2.0000	 &1.7567e-01&	 0.9968\\
	1/$\sqrt{512}$	&1.4633e-03	& 2.0108	 &1.2421e-01&	 1.0002\\
		\bottomrule
	\end{tabular}
	\label{table1}
\end{table}
\begin{table}
	\centering
	\caption{Convergence order in spatial direction for Example \ref{example1} on Voronoi meshes with \(k=2\).}
	\begin{tabular}{ccccc}
		\toprule 
		$h$ & $L^2$-error & $\text{Order}$ & $H^1$-error & $\text{Order}$\\
		\midrule	
		1/$\sqrt{32}$&1.4065e-03&	 	& 6.0044e-02&	 \\
	1/$\sqrt{64}$&	5.0690e-04&	 2.9446	 &2.9608e-02&	 2.0401\\
		1/$\sqrt{128}$&	1.7415e-04&	 3.0828	 &1.4638e-02&	 2.0325\\
	1/$\sqrt{256}$&	6.1541e-05&	 3.0014	 &7.2992e-03&	 2.0079\\
		1/$\sqrt{512}$&	2.1508e-05&	 3.0334	 &3.6139e-03&	 2.0283\\
		\bottomrule
	\end{tabular}
	\label{table2}
\end{table}

\begin{table}
	\centering
	\caption{Convergence order in spatial direction for Example \ref{example1} on non-convex meshes with \(k=1\).}
	\begin{tabular}{ccccc}
		\toprule 
		$h$ & $L^2$-error & $\text{Order}$ & $H^1$-error & $\text{Order}$\\
		\midrule
1/10&9.0974e-03&	 	& 3.2625e-01&	 \\
1/15&4.0876e-03&	 1.9731	 &2.1722e-01&	 1.0032\\
1/20&2.3084e-03&	 1.9863	 &1.6271e-01&	 1.0042\\
1/25&1.4799e-03&	 1.9923	 &1.3005e-01&	 1.0041\\
1/30&1.0286e-03&	 1.9953	 &1.0830e-01&	 1.0038\\
		\bottomrule
	\end{tabular}
	\label{table3}
\end{table}

\begin{table}
	\centering
	\caption{Convergence order in spatial direction for Example \ref{example1} on non-convex meshes with \(k=2\).}
	\begin{tabular}{ccccc}
		\toprule 
		$h$ & $L^2$-error & $\text{Order}$ & $H^1$-error & $\text{Order}$\\
		\midrule
			1/10&3.3693e-04&	 	& 2.6161e-02&	 \\
			1/15&9.8591e-05&	 3.0308	 &1.1612e-02&	 2.0032\\
			1/20&4.1231e-05&	 3.0304	 &6.5277e-03&	 2.0022\\
			1/25&2.0997e-05&	 3.0242	 &4.1762e-03&	 2.0017\\
			1/30&1.2115e-05&	 3.0161	 &2.8993e-03&	 2.0016\\
		\bottomrule
	\end{tabular}
	\label{table4}
\end{table}

\begin{table}
	\centering
	\caption{Convergence order in temporal direction for Example \ref{example1} on rectangle meshes with \(k=2\).}
	\begin{tabular}{ccccc}
		\toprule 
		$h$ & $L^2$-error & $\text{Order}$ & $H^1$-error & $\text{Order}$\\
		\midrule
			1/10&1.0607e-04&	 &	 7.9596e-03&	 \\
			1/15&3.0096e-05&	 3.1069&	 3.5288e-03&	 2.0061\\
			1/20&1.2535e-05&	 3.0445&	 1.9864e-03&	 1.9975\\
			1/25&6.3744e-06&	 3.0305&	 1.2717e-03&	 1.9985\\
			1/30&3.6749e-06&	 3.0209&	 8.8329e-04&	 1.9990\\
		\bottomrule
	\end{tabular}
	\label{table5}
\end{table}
\begin{example}\label{example2}
In our second example, let $\gamma=10^{-4}$ and the initial function is take as 
\begin{align}
u_0(x,y,0) = 0.2\big(\sin(2x)\sin(3y) + \sin(5x)\sin(5y)\big),\quad (x,y) \in [0,2\pi]\times[0,2\pi].\notag 
\end{align}
Obviously, the exact solution for this example is unknown.

This example aims to investigate the energy dissipation property and the evolution of numerical solution for the proposed numerical scheme. Figure \ref{figure2} illustrates the temporal evolution of the numerical solution at different time instants. It can be observed that the numerical solutions change drastically over an extremely short initial period and then evolve slowly thereafter, which is consistent with the numerical results in the existing literature \cite{JiangSunTang2026}. Furthermore, Figure \ref{figure3} depicts the energy decay property of the numerical solution, which is consistent with the preceding theoretical analysis. 

\begin{figure}[!htbp]
	\centering
	\subfigure[Time $t=0.0$.]{\begin{minipage}[t]{0.48\linewidth}
			\includegraphics[width=1\linewidth]{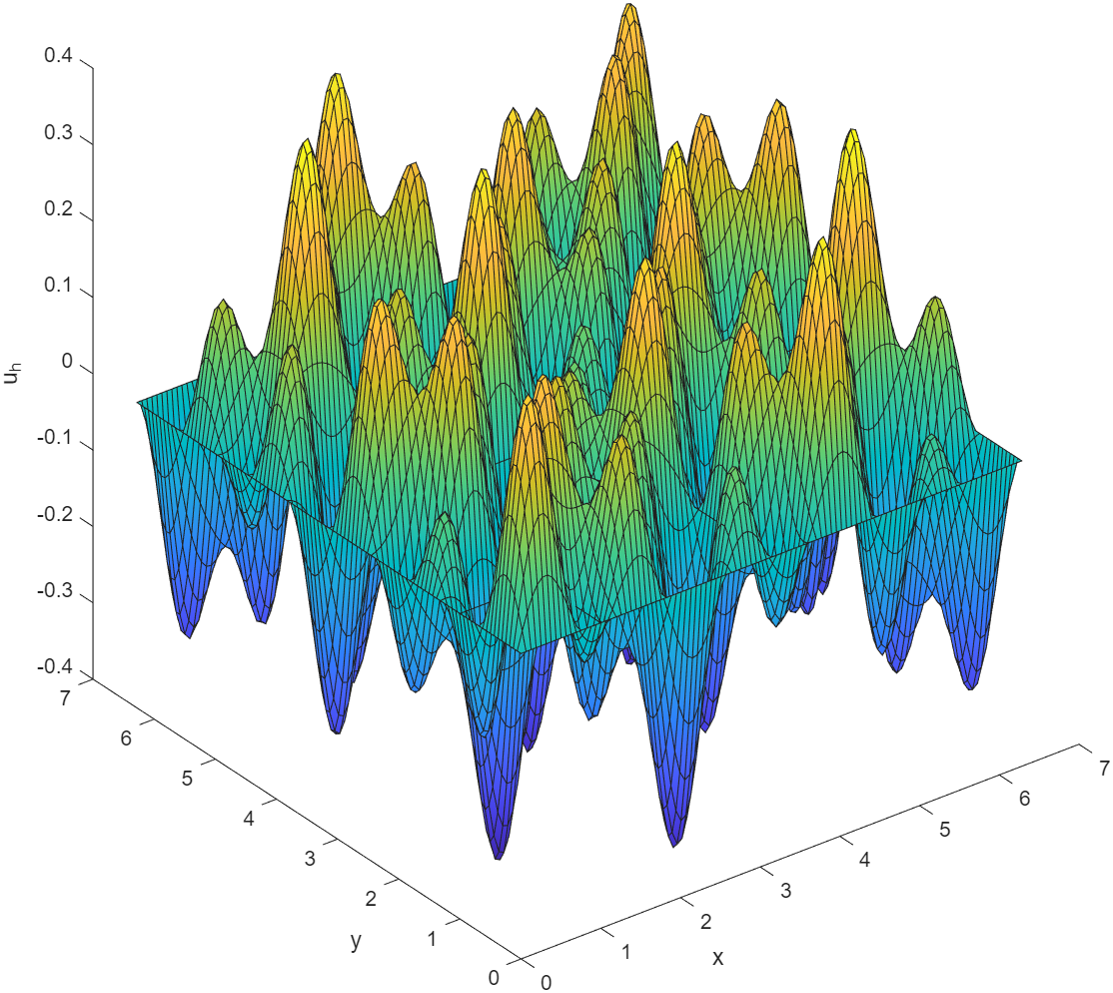}
		\end{minipage}
	}
	\subfigure[Time $t=0.1$. ]{\begin{minipage}[t]{0.48\linewidth}
			\includegraphics[width=1\linewidth]{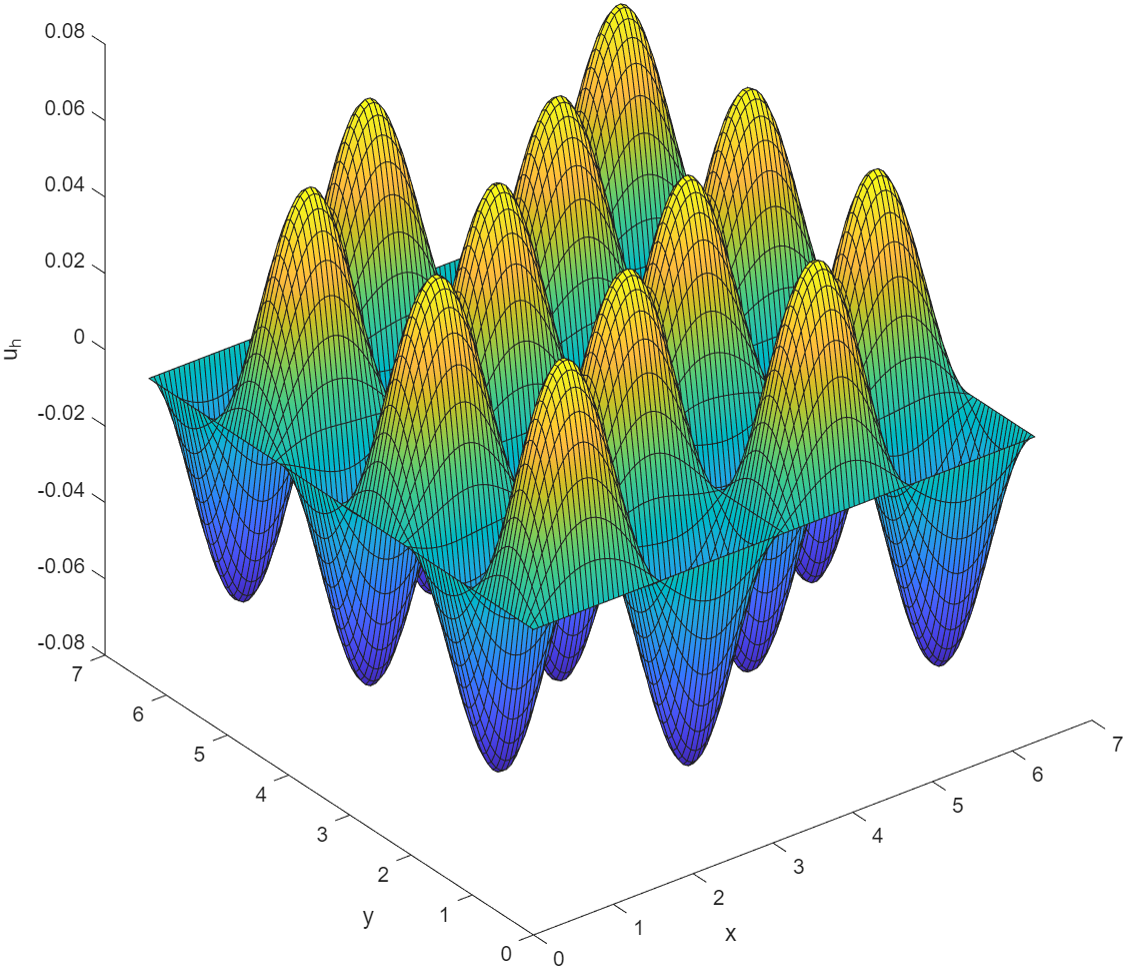}
		\end{minipage}
	}
		\subfigure[Time $t=0.5$. ]{\begin{minipage}[t]{0.48\linewidth}
			\includegraphics[width=1\linewidth]{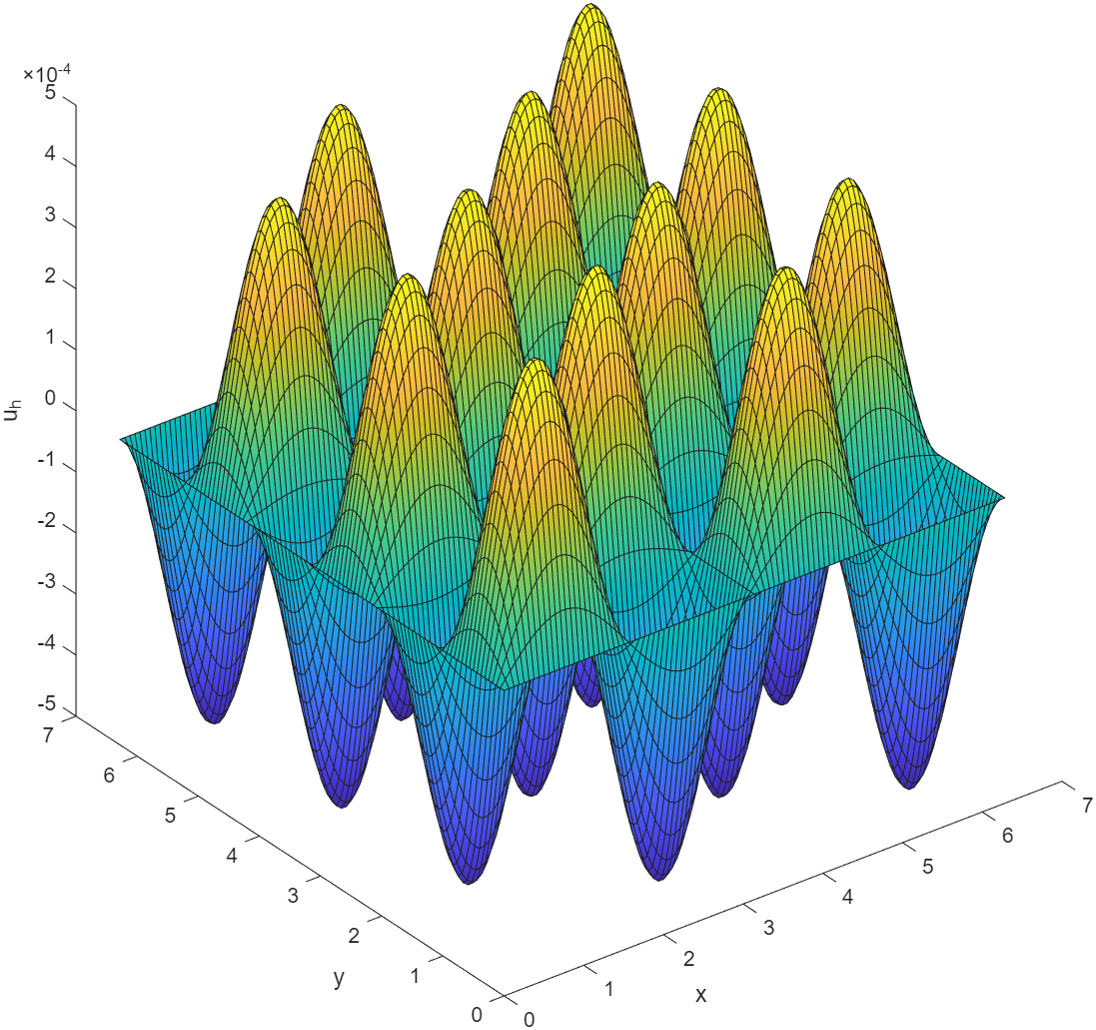}
		\end{minipage}
	}
		\subfigure[Time $t=0.8$. ]{\begin{minipage}[t]{0.48\linewidth}
			\includegraphics[width=1\linewidth]{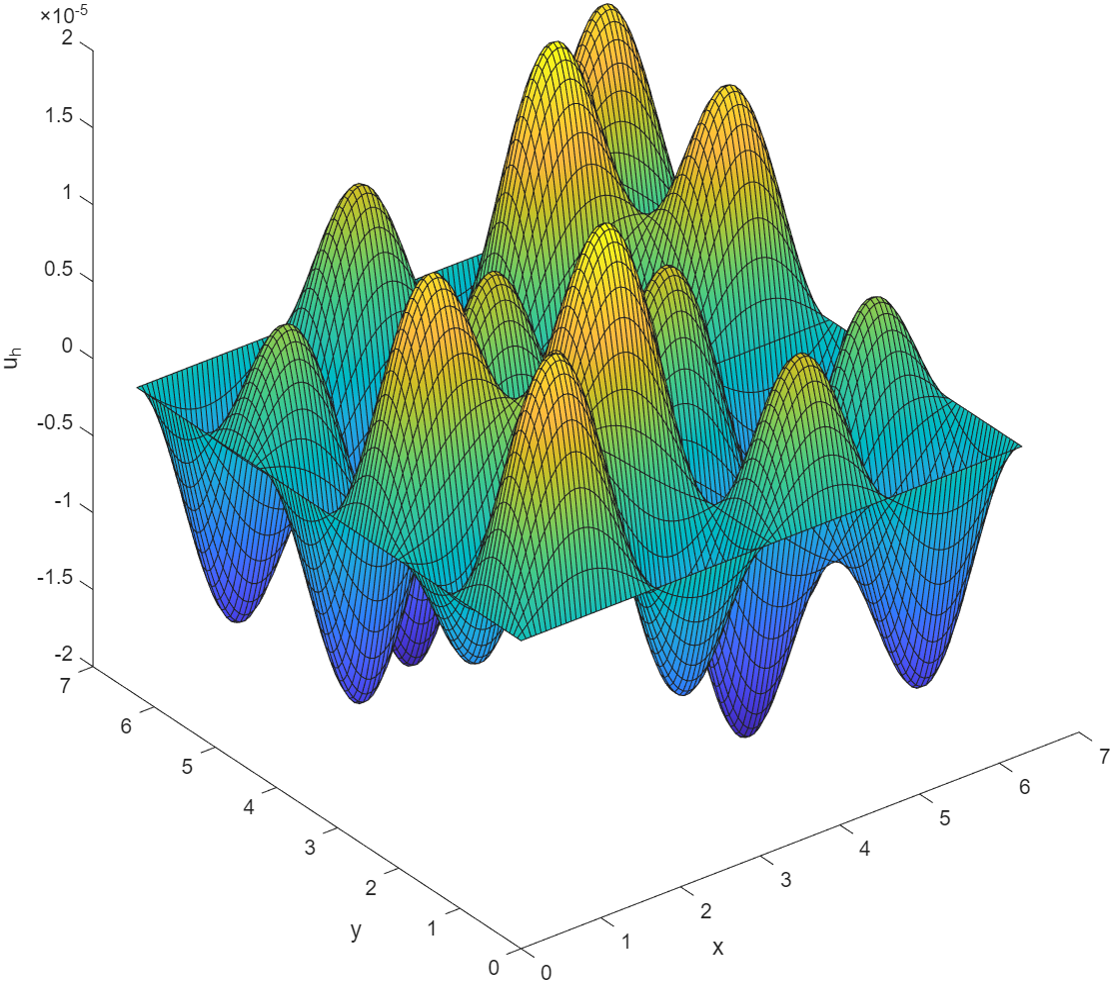}
		\end{minipage}
	}
		\subfigure[Time $t=1.0$. ]{\begin{minipage}[t]{0.48\linewidth}
			\includegraphics[width=1\linewidth]{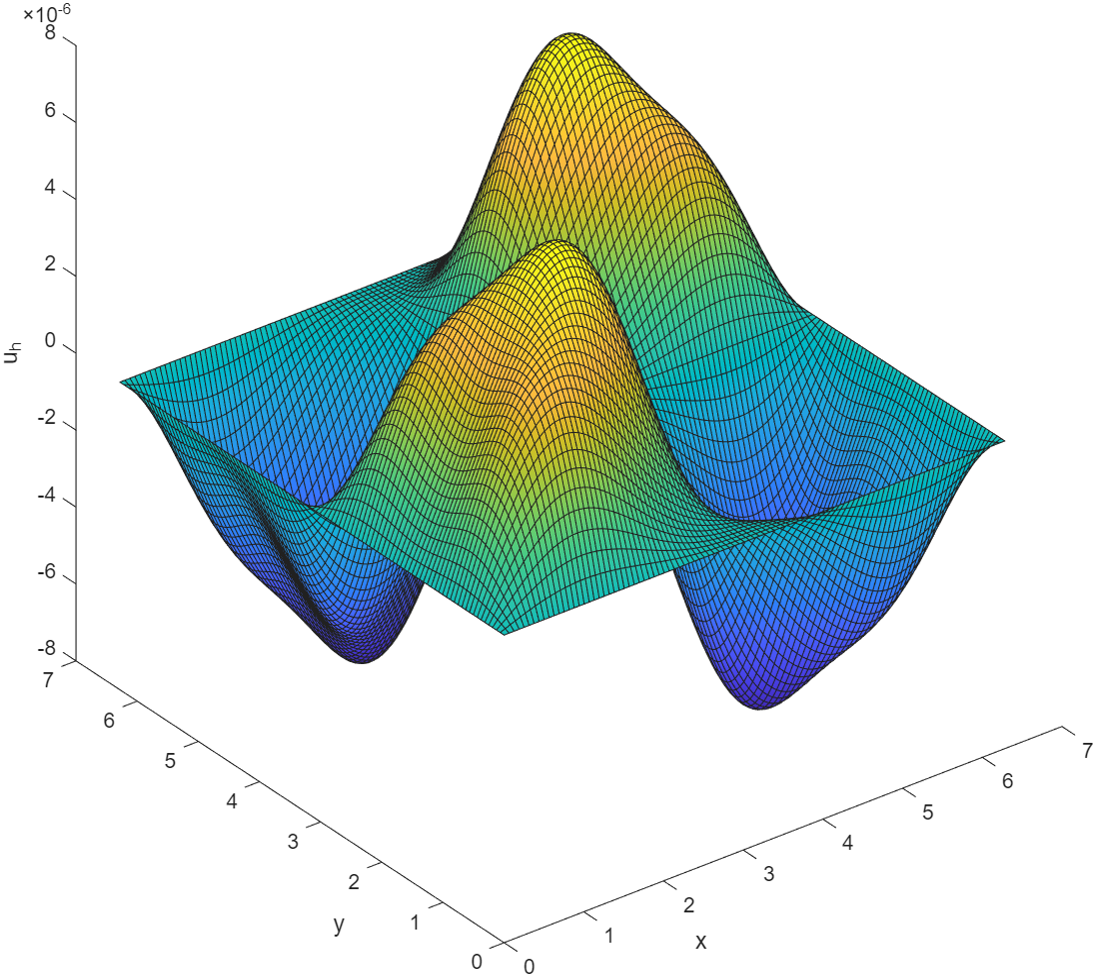}
		\end{minipage}
	}
		\subfigure[Time $t=2.0$. ]{\begin{minipage}[t]{0.48\linewidth}
			\includegraphics[width=1\linewidth]{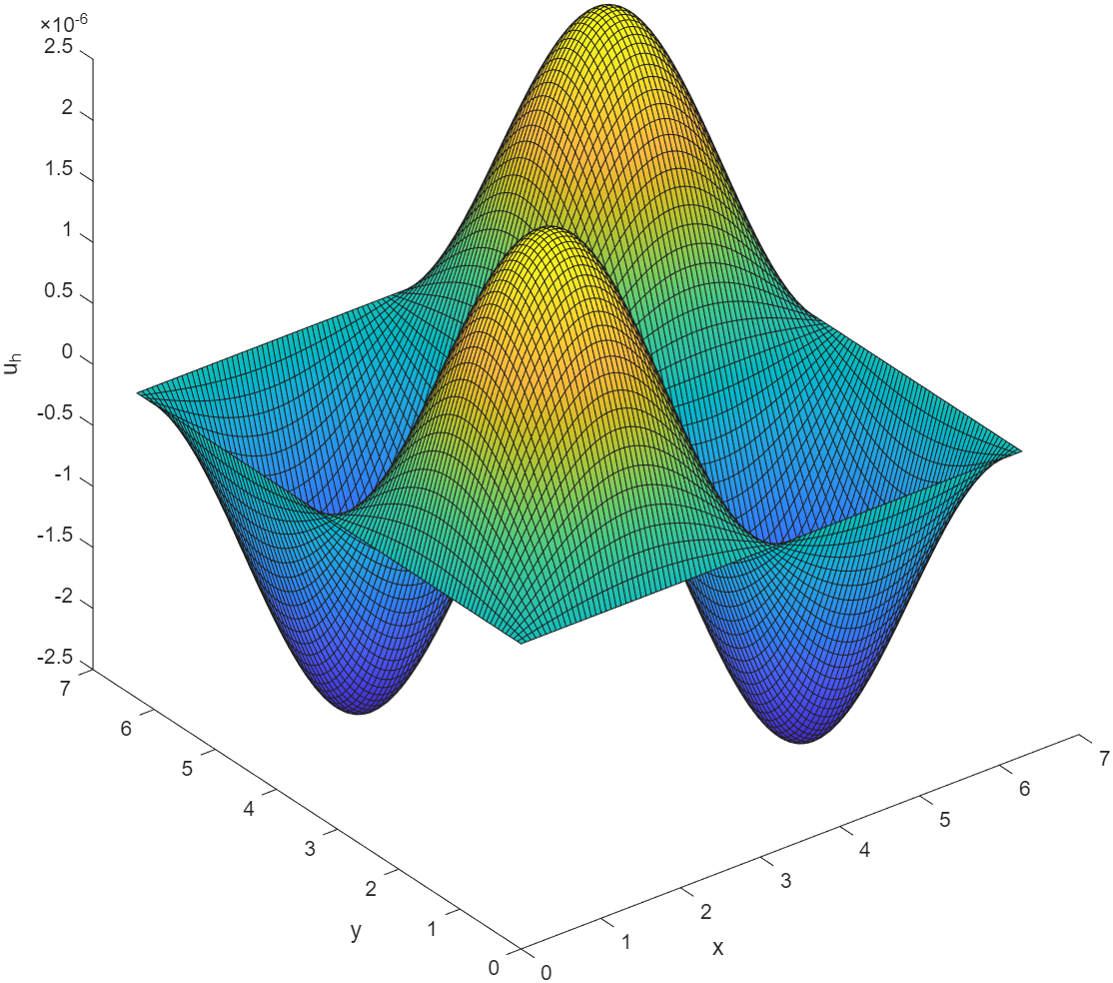}
		\end{minipage}
	}
	\caption{Evolution of numerical solution for Example \ref{example2} at distinct time levels.}
	\label{figure2}
\end{figure}
\begin{figure}[!htbp]
	\centering
	\includegraphics[width=7cm]{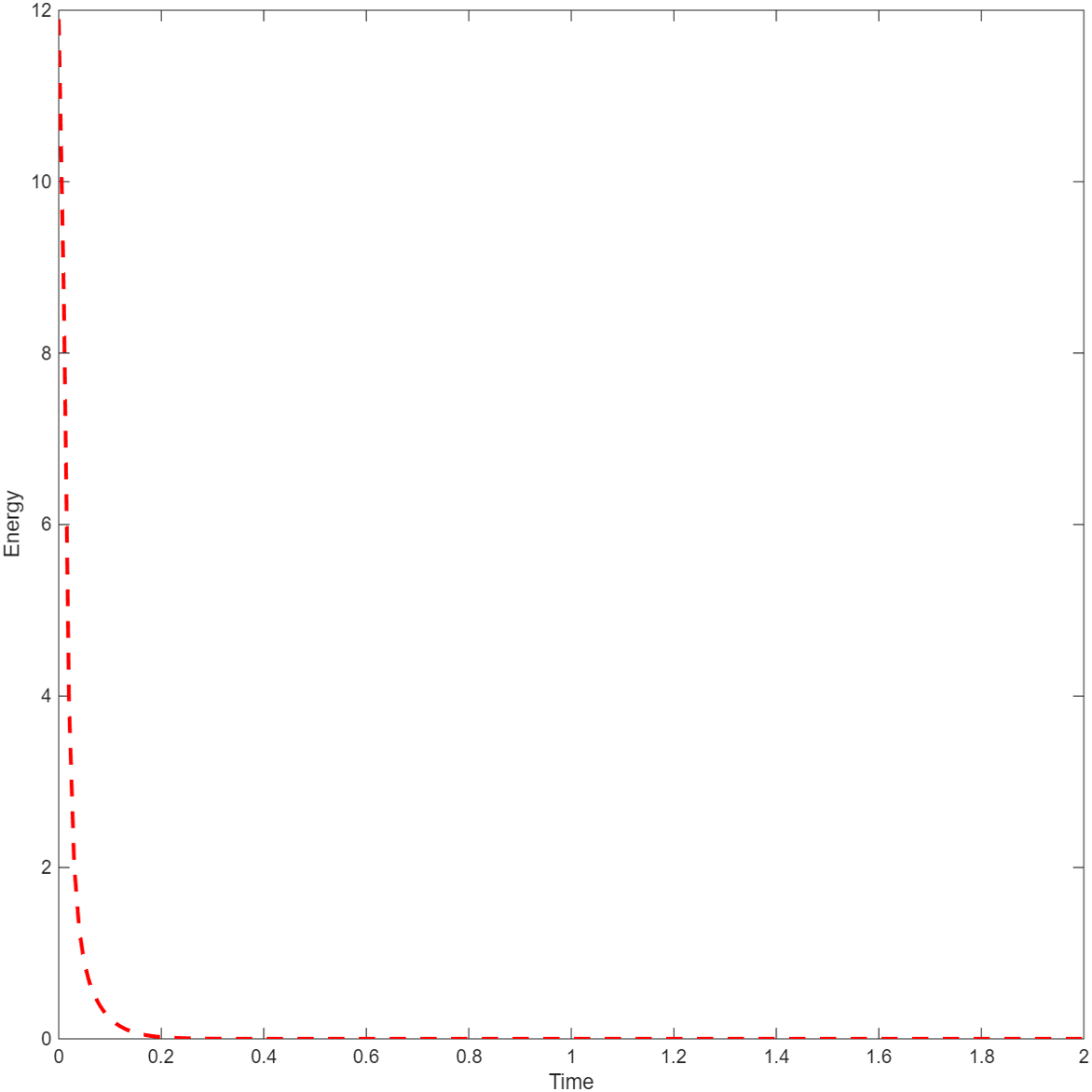}
	\caption{Dissipative property of discrete energy over time.}\label{figure3}
\end{figure}
\end{example}
\section{Conclusions}\label{section5}
This paper constructs a fully discrete mixed virtual element algorithm for the two-dimensional extended Fisher–Kolmogorov (EFK) equation by adopting the leapfrog temporal discretization scheme. It can preserve the energy dissipation property of the original equation. Through elaborate treatment of the nonlinear term and the inverse inequality technique, we strictly prove the unconditional optimal convergence of the fully discrete numerical scheme. On the basis of this work, several research directions deserve further investigation. First, we only derive \(L^2\)-norm convergence estimates for the primary variables, without providing corresponding error bounds in the \(H^1\)- norm. Besides, we do not establish error estimates for the intermediate variable $v$. All these points merit further study in future work. Second, higher-order temporal discretization schemes can be considered in future work, such as Runge–Kutta methods and time-discontinuous discretization techniques.

\section*{Declaration of Interest Statement}
The authors declare no competing financial interests or personal affiliations that might affect the findings presented in this paper.

\section*{Credit Author Statement}
{\bf{Zhen Guan:}} Methodology, Software, Validation, Formal analysis, Writing-Original Draft, Project administration, Funding acquisition; {\bf{Xianxian Cao:}} Conceptualization, Methodology, Validation, Formal analysis, Resources, Writing-Review \& Editing, Project administration, Funding acquisition; {\bf{Houchao Zhang:}} Methodology; {\bf{Junjun Wang:}} Methodology.

\section*{Data availability}
Data will be made available on request.

\section*{Acknowledgments}
This work is supported by the Doctoral Starting Foundation of Pingdingshan University (No. PXY-BSQD2023022) and the Natural Science Foundation of Henan Province (Nos. 242300420655, 262300420348).

\end{document}